%% file: krylov_paper.tex
\pdfoutput=1
\documentclass[11pt]{article} 

\newcommand{\mathprog}[1]{foo}
\newcommand{\arxiv}[1]{ba}
\newcommand{\nips}[1]{bar}
\ifdefined \SubmitToMathprog
\renewcommand{\mathprog}[1]{#1}%
\renewcommand{\arxiv}[1]{}%
\renewcommand{\nips}[1]{}%
\else%
\ifdefined \SubmitToNIPS
\renewcommand{\mathprog}[1]{}%
\renewcommand{\arxiv}[1]{}%
\renewcommand{\nips}[1]{#1}%
\else
\renewcommand{\mathprog}[1]{}%
\renewcommand{\arxiv}[1]{#1}%
\renewcommand{\nips}[1]{}%
\fi%
\fi%

\arxiv{\usepackage[top=1in, right=1in, left=1in, bottom=1in]{geometry}}

\usepackage[numbers,square]{natbib}
\nips{
\PassOptionsToPackage{numbers, square}{natbib}
\usepackage[final]{neurips_2018}
}

\usepackage{amsthm}
\usepackage[hidelinks]{hyperref}

\usepackage{thmtools}
\usepackage{thm-restate}

\theoremstyle{plain}
\newtheorem{theorem}{Theorem}
\newtheorem*{theorem*}{Theorem}
\newtheorem{lemma}{Lemma}

\newtheorem*{lemma*}{Lemma}
\newtheorem{corollary}[theorem]{Corollary}

\input{macros.sty}

\usepackage{enumitem}

\nips{
\title{Analysis of Krylov Subspace Solutions of 
Regularized Nonconvex Quadratic Problems}

\author{
	Yair Carmon\\
	Department of Electrical Engineering\\
	Stanford University\\
	\texttt{yairc@stanford.edu} \\
	\And
	John C.\ Duchi\\
	Departments of Statitstics and Electrical Engineering\\
	Stanford University\\
	\texttt{jduchi@stanford.edu} \\
}
}

\arxiv{
	\title{Analysis of Krylov Subspace Solutions \\ of 
		Regularized Nonconvex Quadratic Problems}
	
	\author{Yair Carmon ~~~ John C.\ Duchi\\
		\texttt{\{\href{mailto:yairc@stanford.edu}{yairc},%
			\href{mailto:jduchi@stanford.edu}{jduchi}\}@stanford.edu}}
	\date{}
}

\begin{document}
\maketitle

\begin{abstract}
We provide convergence rates for Krylov subspace solutions
to the trust-region and cubic-regularized (nonconvex) quadratic problems. 
Such solutions may be efficiently computed by the Lanczos method and 
have long been used in practice. We prove error bounds of the 
form $1/t^2$ and $e^{-4t/\sqrt{\kappa}}$, where $\kappa$ is a  
condition number for the problem, and $t$ is the Krylov subspace order 
(number of Lanczos iterations). We also
provide lower bounds showing that our analysis is sharp.
\end{abstract}

\input{intro.tex}
\input{tr.tex}
\input{cubic.tex}

\input{lower.tex}

\input{experiment.tex}

\newpage 

\section*{Acknowledgments}
We thank the anonymous reviewers for several helpful questions and 
suggestions. Both authors were supported by NSF-CAREER Award 1553086 
and the Sloan Foundation. YC was partially supported by the Stanford 
Graduate Fellowship.

\setlength{\bibsep}{6pt}
\bibliographystyle{abbrvnat}
\bibliography{bib}

\newpage

\appendix

\part*{Supplementary material}

\input{lanczos.tex}
\input{polys.tex}

\input{ub-proofs.tex}

\input{lb-proofs.tex}
\input{experiment-details.tex}

\end{document}

%% file: intro.tex
\section{Introduction}

Consider the potentially nonconvex quadratic function
\begin{equation*}
  \f(x) \defeq \half x^T A x + b^T x,
\end{equation*}
where $A \in \R^{d \times d}$ and $b \in \R^d$.
We wish to solve regularized minimization 
problems of the form
\begin{equation}\label{eq:problems}
  \minimize_{x}
  \f(x) ~ \subjectto \norm{x} \le R
  ~~~\mbox{and}~~~
  \minimize_{x} \f(x) + \frac{\rho}{3} \norm{x}^3,
\end{equation}
where $R$ and $\rho \ge 0$ are regularization parameters.  These problems
arise primarily in the family of trust-region and cubic-regularized Newton
methods for general nonlinear optimization
problems~\cite{ConnGoTo00,NesterovPo06, Griewank81,CartisGoTo11},
which optimize a smooth function $g$ by sequentially
minimizing local models of the form
\begin{equation*}
  g(x_i + \Delta)
  \approx 
  g(x_i) + \grad g(x_i)^T \Delta + \half \Delta^T \hess g(x_i) 
  \Delta = g(x_i) + \f[\hess g(x_i), \grad g(x_i)](\Delta),
\end{equation*}
where $x_i$ is the current iterate and $\Delta \in \R^d$ is the search
direction. Such models tend to be unreliable for large $\norm{\Delta}$,
particularly when $\hess g(x_i)\nsucc 0$.  Trust-region and cubic
regularization methods address this by constraining and regularizing the
direction $\Delta$, respectively. 

Both classes of methods and their 
associated subproblems are the
subject of substantial ongoing 
research~\cite{HazanKo16,Ho-NguyenKi16,CarmonDu16,AgarwalAlBuHaMa17,LendersKiPo18}.
 In 
the machine learning community, there is growing interest in using 
these methods for minimizing (often nonconvex) training losses, handling 
the large finite-sum structure of learning problems  by means of 
sub-sampling~\cite{ReigerJoMcAu17,KohlerLu17,BlanchetCaMeSc16,
	YaoXuRoMa18,TripuraneniStJiReJo17}.

The problems~\eqref{eq:problems} are challenging to solve in
high-dimensional settings, where direct decomposition (or
even storage) of the matrix $A$ is infeasible. In some scenarios, however,
computing matrix-vector products $v\mapsto Av$ is feasible. Such is the case
when $A$ is the Hessian of a neural network, where $d$ may be in the
millions and $A$ is dense, and yet we can compute Hessian-vector products
efficiently on batches of training
data~\cite{Pearlmutter94,Schraudolph02}. 

In this paper we consider a scalable approach for approximately 
solving~\eqref{eq:problems}, which consists of minimizing the objective 
in the \emph{Krylov subspace} of order $t$, 
\begin{equation}\label{eq:krylov-def}
  \mc{K}_t(A,b) \defeq 
  \mathrm{span}\{b, Ab, \ldots, A^{t-1}b\}.
\end{equation}
This requires only $t$ matrix-vector products,
and the Lanczos method allows one to efficiently find the solution to
problems~\eqref{eq:problems} over $\mc{K}_t(A,b)$ (see,
e.g.~\cite[Sec.~2]{GouldLuRoTo99,CartisGoTo11}).  Krylov subspace 
methods are familiar in numerous large-scale numerical problems, 
including
conjugate gradient methods, eigenvector problems, or solving linear
systems~\cite{HestenesSt52,
  Nemirovski94,TrefethenBa97,GolubVa89}. 

It is well-known that, with exact arithmetic, the order $d$ subspace
$\Krylov[d]$ generically contains the global solutions
to~\eqref{eq:problems}. However, until recently the literature contained no
guarantees on the rate at which the suboptimality of the solution approaches
zero as the subspace dimension $t$ grows. This is in contrast to the two
predominant Krylov subspace method use-cases---convex quadratic
optimization~\cite{GolubVa89,NemirovskiYu83,Nesterov04} and eigenvector
finding~\cite{KuczynskiWo92}---where such rates of convergence have 
been
known for decades.  \citet{ZhangShLi17} make substantial
progress on this gap,
establishing bounds implying a linear rate of convergence for the 
trust-region variant of problem~\eqref{eq:problems}.

In this work we complete the picture, proving that the optimality gap of the 
order $t$ Krylov subspace solution to either of the 
problems~\eqref{eq:problems}  is bounded by both $e^{-4t/\sqrt{\kappa}}$ 
and $t^{-2}\log^2(\norms{b}/|u_{\min}^T b|)$. 
Here $\kappa$ is a condition number for the problem that naturally 
generalizes the classical condition number of the matrix $A$, and 
$u_{\min}$ is an eigenvector of $A$ corresponding to its smallest 
eigenvalue. Using randomization, we may replace $|u_{\min}^T b|$ with a 
term proportional to $1/\sqrt{d}$, circumventing the well-known ``hard 
case'' of the problem~\eqref{eq:problems} (see 
Section~\ref{sec:upper-rand}).
Our analysis both leverages and 
unifies the known results for convex quadratic and eigenvector problems, 
which constitute special cases of~\eqref{eq:problems}. 

%
%
%
%
%
%
%
%
%
%
%
%
%
%
%
%
%
%
%
%
%
%
%
%
%
%
%
%
%
%
%
%
%
%

\paragraph{Related work}

\citet{ZhangShLi17} show that the error of certain polynomial 
approximation problems bounds the suboptimality of Krylov subspace 
solutions to the trust region-variant of the problems~\eqref{eq:problems}, 
implying convergence at a rate exponential in $-t/\sqrt{\kappa}$. Based on 
these bounds, the authors propose novel stopping criteria for subproblem 
solutions in the trust-region optimization method, showing good empirical 
results. However, the bounds
of~\cite{ZhangShLi17} become weak for large $\kappa$ and vacuous in the
hard case where $\kappa=\infty$.

Prior works develop algorithms for solving~\eqref{eq:problems} 
with convergence guarantees that hold in the hard case. 
\citet{HazanKo16}, \citet{Ho-NguyenKi16},
and \citet{AgarwalAlBuHaMa17} propose algorithms that obtain error  
roughly $t^{-2}$ after computing $t$ matrix-vector products.
The different algorithms these papers propose all essentially reduce the
problems~\eqref{eq:problems} to a sequence of eigenvector and convex
quadratic problems to which standard algorithms
apply. In previous work~\cite{CarmonDu16}, we analyze gradient 
descent---a direct, local
method---for the cubic-regularized problem. There, we show a rate of 
convergence
roughly $t^{-1}$, reflecting the well-known complexity gap between gradient
descent (respectively, the power method) and conjugate gradient
(respectively, Lanczos) methods~\cite{TrefethenBa97,GolubVa89}.

Our development differs from this prior work in the following ways.
\begin{enumerate}[leftmargin=*]
\item We analyze a practical approach, implemented in efficient optimization
  libraries~\cite{GouldOrTo03,LendersKiPo18}, with essentially no tuning
  parameters. Previous algorithms~\cite{HazanKo16,Ho-NguyenKi16,
    AgarwalAlBuHaMa17} are convenient for theoretical analysis but less
  conducive to efficient implementation; each has several parameters that
  require tuning, and we are unaware of numerical experiments with any of
  the approaches.

\item We provide both linear ($e^{-4t/\sqrt{\kappa}})$ and sublinear
  ($t^{-2}$) convergence guarantees. In contrast, the
  papers~\cite{HazanKo16,Ho-NguyenKi16,AgarwalAlBuHaMa17} provide only a
  sublinear rate; \citet{ZhangShLi17} provide only the linear rate.

\item Our analysis applies to both the trust-region and cubic regularization 
  variants in~\eqref{eq:problems}, 
  while~\cite{HazanKo16,Ho-NguyenKi16,ZhangShLi17} consider only the  
  trust-region problem, and \cite{ZhangShLi17,CarmonDu16} consider only 
  cubic regularization.

\item We provide lower bounds---for adversarially constructed problem
  instances---showing our convergence guarantees are tight to within
  numerical constants. By a resisting oracle
  argument~\cite{NemirovskiYu83}, these bounds apply to any
  deterministic algorithm that accesses $A$
  via matrix-vector products.
  
\item Our arguments are simple and transparent, and we
  leverage established  
  results on convex optimization and the eigenvector problem to give 
  short proofs of our main results.
\end{enumerate}

\paragraph{Paper organization} 
In Section~\ref{sec:upper} we state and prove our convergence rate 
guarantees for the trust-region problem. 
Then, in Section~\ref{sec:upper-cr} we quickly transfer those results to the 
cubic-regularized problem by showing that it always has a smaller 
optimality gap. Section~\ref{sec:lower} gives our lower bounds, stated for 
cubic regularization but immediately applicable to the trust-region 
problem by the same optimality gap bound. Finally, in Section~\ref{sec:exp} 
we illustrate our analysis with some numerical experiments.

\newcommand{\lmin}{\lambda_{\min}}
\newcommand{\lmax}{\lambda_{\max}}

\paragraph{Notation}
For a symmetric matrix $A\in\R^{d\times d}$ and vector $b$ we let 
$\f(x) \defeq \half x^T A x + b^T x.$
We let $\lmin(A)$ and $\lmax(A)$ denote the minimum and maximum 
eigenvalues of $A$, and let $u_{\min}(A), u_{\max}(A)$ denote their 
corresponding (unit) eigenvectors, dropping the argument $A$ when clear 
from context. For integer $t\ge1$ we let
$\Polys[t] \defeq \left\{  c_0 + c_1 x + \cdots + c_{t-1}x^{t-1}\mid 
c_i\in \R \right\}$ 
be the polynomials of degree at most $t-1$, so that the Krylov 
subspace~\eqref{eq:krylov-def} is $\Krylov = \left\{ p(A)b \mid 
p\in\Polys[t] \right\}$.
We 
use $\norm{\cdot}$ to denote Euclidean norm on $\R^d$ and 
$\ell_2$-operator norm on $\R^{d\times d}$. Finally, we denote 
$(z)_+\defeq\max\{z,0\}$ and $(z)_- \defeq \min\{z,0\}$.

%% file: tr.tex
\section{The trust-region problem}\label{sec:upper}

Fixing a symmetric matrix $A\in \R^{d\times d}$, vector $b\in\R^d$ and 
trust-region radius $R>0$, we let
\begin{equation*}
\str \in \argmin_{x\in\R^d,~\norm{x}\le R} \f(x) = \half x^T A x + 
b^T x
\end{equation*}
denote a solution (global minimizer) of the trust region problem. Letting 
$\lmin,\lmax$ denote the extremal eigenvalues of $A$, $\str$ 
admits the following characterization~\cite[Ch.~7]{ConnGoTo00}:
$\str$ solves problem~\eqref{eq:problems} if and only if
there exists $\ltr$ such that
\begin{equation}
  (A + \ltr I) \str = -b,
  ~~~ \ltr \ge (-\lmin)_+,
  ~~~ \mbox{and} ~~~
  \ltr(R - \norm{\str}) = 0.
  \label{eq:tr-optimality}
\end{equation}
The optimal Lagrange multiplier $\ltr$ always exists and is unique,
and if $\ltr > -\lmin$ the solution $\str$ is unique and satisfies 
$\str = -(A+\ltr I )^{-1}b$. Letting $u_{\min}$ denote the eigenvector of 
$A$ corresponding to $\lmin$, the
characterization~\eqref{eq:tr-optimality}
shows that $u_{\min}^T b \ne 0$ implies 
$\ltr > -\lmin$.

Now, consider the Krylov subspace solutions, and for $t>0$, let
\begin{equation*}
  \itertr_t \in \argmin_{x\in\Krylov,~\norm{x}\le R} 
  \f(x) = \half x^T A x + b^T x
\end{equation*}
denote a minimizer of the trust-region problem in the Krylov subspace of
order $t$ . \citet{GouldLuRoTo99} show how to compute the
Krylov subspace solution $\itertr_t$ in time dominated by the cost of
computing $t$ matrix-vector products using the Lanczos method
(see also Section~\ref{sec:lanczos} of the supplement).

\subsection{Main result}
With the notation established above, our main result follows.
\begin{restatable}{theorem}{thmMain}\label{thm:tr}
  For every $t > 0$,
  \begin{equation}\label{eq:tr-lin-time}
    \f(\itertr_t) - \f(\str) \le 36\left[\f(0) - \f(\str) \right]
    \exp\left\{
    -4t\sqrt{\frac{\lmin + \ltr}{\lmax +\ltr}}
    \right\},
  \end{equation}
  and
  \begin{equation}\label{eq:tr-sublin-time}
    \f(\itertr_t) - \f(\str) \le \frac{(\lmax - 
      \lmin)\norm{\str}^2}{(t-\half)^2} 
    \left[ 
      4 + \frac{\I_{\{\lmin < 0\}}}{8}
      \log^2\left(\frac{4\norm{b}^2}{(u_{\min}^T b)^2}\right)
      \right].
  \end{equation}
\end{restatable}

Theorem~\ref{thm:tr} characterizes two convergence regimes: 
linear~\eqref{eq:tr-lin-time} and sublinear~\eqref{eq:tr-sublin-time}. 
Linear convergence occurs when $t \gtrsim \sqrt{k}$, where $\kappa = 
\frac{\lmax + \ltr}{\lmin + \ltr} \ge 1$ is the condition 
number for the problem. There, the error 
decays exponentially and falls beneath $\epsilon$ in roughly 
$\sqrt{\kappa}\log{\frac{1}{\epsilon}}$ Lanczos iteration. Sublinear 
convergence occurs when $t \lesssim \sqrt{k}$, and there the error decays 
polynomially and falls beneath $\epsilon$ in roughly 
$\frac{1}{\sqrt{\epsilon}}$ iterations. For worst-case problem instances 
this characterization is tight to constant factors, as we show in 
Section~\ref{sec:lower}.

The guarantees of Theorem~\ref{thm:tr} closely resemble the well-known
guarantees for the conjugate gradient method~\cite{TrefethenBa97},
including them as the special case $R = \infty$ and $\lmin \ge 0$. For convex
problems, the radius constraint $\norm{x}\le R$ always improves the
conditioning of the problem, as $\frac{\lmax}{\lmin} \ge
\frac{\lmax+\ltr}{\lmin+\ltr}$; the smaller $R$ is, the better conditioned
the problem becomes. For non-convex problems, the sublinear rate features an
additional logarithmic term that captures the role of the eigenvector
$u_{\min}$. The first rate~\eqref{eq:tr-lin-time}
is similar to those of \citet[Thm.~4.11]{ZhangShLi17}, though
with somewhat more explicit dependence on $t$.

In the ``hard case,'' which corresponds to $u_{\min}^T b = 0$ and $\lmin +
\ltr = 0$ (cf.~\cite[Ch.~7]{ConnGoTo00}), both the bounds in
Theorem~\ref{thm:tr} become vacuous, and indeed $\itertr_t$ may not converge
to the global minimizer in this case. However, as the
bound~\eqref{eq:tr-sublin-time} depends only logarithmically on $u_{\min}^T
b$, it remains valid even extremely close to the hard case. In
Section~\ref{sec:upper-rand} we describe two simple randomization 
techniques
with convergence guarantees that are valid in the hard case as well.

\subsection{Proof sketch}\label{sec:upper-proof-outline}
Our analysis reposes on two elementary observations. First, we note that  
Krylov subspaces are invariant to shifts by scalar matrices, \ie 
$\mc{K}_t(A,b) = \mc{K}_t(A_\lambda, b)$ for any $A,b,t$ where 
$\lambda\in\R$, and 
\begin{equation*}
A_\lambda \defeq A + \lambda I. 
\end{equation*}
Second, we observe 
that  for every point $x$ and $\lambda\in\R$
\begin{align}\label{eq:tr-gamma-pivot-outline}
\f(x) - \f(\str)  & = \f[A_{\lambda},b](x) - \f[A_{\lambda},b](\str) + 
\frac{\lambda}{2}(\norm{\str}^2 - \norm{x}^2) 
 \end{align}
Our strategy then is to choose $\lambda$ such that $A_\lambda \succeq 
0$, 
and then use known results to find $y_t \in \Krylov[t][A_\lambda,b] = 
\Krylov[t][A,b]$ that rapidly reduces the ``convex error'' term 
$\f[A_{\lambda},b](y_t) - \f[A_{\lambda},b](\str)$. 
We then adjust $y_t$ to obtain a feasible 
point $x_t$ such that the ``norm error'' term 
$\frac{\lambda}{2}(\norm{\str}^2 - \norm{x_t}^2)$ is small.
To establish linear convergence, we take $\lambda=\ltr$ and adjust the 
norm of $y_t$ by taking $x_t=(1-\alpha)y_t$ for some small $\alpha$ that 
guarantees $x_t$ is feasible and that the ``norm error'' term is small. To 
establish sublinear convergence we set $\lambda=-\lmin$ and take $x_t = 
y_t + \alpha \cdot z_t$, where $z_t$ is an approximation for $u_{\min}$ 
within $\Krylov$, and $\alpha$ is chosen to make 
$\norm{x_t}=\norm{\str}$. This means the ``norm error'' vanishes, 
while the ``convex error'' cannot increase too much, as $A_{-\lmin}z_t 
\approx A_{-\lmin}u_{\min}=0$.

Our approach for proving the sublinear rate of convergence is inspired 
by~\citet{Ho-NguyenKi16}, who also rely on Nesterov's method in 
conjunction 
with Lanczos-based eigenvector approximation. The analysis 
in~\cite{Ho-NguyenKi16} uses an algorithmic reduction, proposing to apply 
the Lanczos method (with a random vector instead of $b$) to approximate 
$u_{\min}$ and $\lmin$, then run Nesterov's method on an approximate 
version of the ``convex error'' term, and then use the approximated 
eigenvector to adjust the norm of the result. We instead argue that all the 
ingredients for this reduction already exist in the Krylov subspace 
$\Krylov$, obviating the need for explicit eigenvector estimation or 
 actual application of accelerated gradient descent.

\subsection{Building blocks}
Our proof uses the following classical results.
\begin{restatable}[Approximate matrix 
inverse]{lemma}{lemLinCheby}\label{lem:lin-cheby}
	Let $\alpha,\beta$ satisfy $0 < \alpha \le \beta$, and let $\kappa = 
	\beta/\alpha$. For any $t\ge1$ there exists a polynomial $p$ 
	of degree at most $t-1$, such that for every $M$ satisfying $\alpha I 
	\preceq M \preceq \beta I$,
	\begin{equation*}
	\opnorm{I - Mp(M)} \le 2e^{-2t/\sqrt{\kappa}}.
	\end{equation*}
\end{restatable}
\begin{restatable}[Convex trust-region 
problem]{lemma}{lemAGD}\label{lem:agd}
	Let $t\ge1$, $M\succeq 0$, $v\in \R^d$ and $r\ge 0$, and let 
	$\f[M,v](x) 
	=\half x^T M x + v^T x$. There exists $x_t \in 
	\Krylov[t][M,v]$ such that 
	\begin{equation*}
	\norm{x_t} \le r
	~~\mbox{and}~~
	\f[M,v](x_t) - \min_{\norm{x}\le r} \f[M,v](x) 
	\le \frac{4\lambda_{\max}(M) \cdot r^2}{(t+1)^2}.
	\end{equation*}
\end{restatable}
\begin{restatable}[{Finding 
	eigenvectors, \cite[Theorem 
	4.2]{KuczynskiWo92}}]{lemma}{lemEigen}\label{lem:eigenvec-tight}
	Let $M\succeq 0$ be such that $u^T M u = 0$ for some unit vector 
	$u\in\R^d$, and let 
	$v\in\R^d$. For every $t\ge1$ there exists $z_t \in \Krylov[t][M,v]$ such 
	that
	\begin{equation*}
	\norm{z_t} = 1
	~~\mbox{and}~~
	z_t^T M z_t \le \frac{\opnorm{M}}{16(t-\half)^2}
	\log^2\left(-2+4\frac{\norm{v}^2}{(u^T v)^2}\right).
	\end{equation*}
\end{restatable}
While these lemmas are standard, their explicit forms are
useful, and we prove them in Section~\ref{sec:prelim-proofs} in the
supplement.  Lemmas~\ref{lem:lin-cheby} and~\ref{lem:eigenvec-tight} are
consequences of uniform polynomial approximation results (cf.\ 
supplement,
Sec.~\ref{sec:polys}). To prove Lemma~\ref{lem:agd} we invoke Tseng's
results on a variant of Nesterov's accelerated gradient
method~\cite{Tseng08}, arguing that its
iterates lie in the Krylov subspace.

\subsection{Proof of Theorem~\ref{thm:tr}}

\paragraph{Linear convergence}\label{sec:upper-lin-proof}
Recalling the notation $\Atr = A + \ltr I$, let $y_t = -p(\Atr)b = p(\Atr)\Atr 
\str$, for the $p\in\mc{P}_t$ which 
Lemma~\ref{lem:lin-cheby} guarantees to satisfy $\opnorm{p(\Atr)\Atr - 
	I}\le 2e^{-2t/\sqrt{\kappa(\Atr)}}$. Let
\begin{equation*}
x_t = (1-\alpha) y_t,
~\mbox{where}~
\alpha = \frac{\norm{y_t} - \norm{\str}}{\max\{\norm{\str},\norm{y_t}\}},
\end{equation*}
so that we are guaranteed $\norm{x_t}\le\norm{\str}$ for any value of 
$\norm{y_t}$. Moreover
\begin{equation*}
|\alpha| = 
\frac{|\norm{y_t}-\norm{\str}|}{\max\{\norm{\str},\norm{y_t}\}} \le 
\frac{\norm{y_t - \str}}{\norm{\str}} =
\frac{\norm{(p(\Atr)\Atr-I)\str}}{\norm{\str}} 
\le 2e^{-2t/\sqrt{\kappa(\Atr)}},
\end{equation*}
where the last transition used $\opnorm{p(\Atr)\Atr - 
	I}\le 2e^{-2t/\sqrt{\kappa(\Atr)}}$.

Since $b = -\Atr\str$, we have 
$\f[\Atr,b](x) = \f[\Atr,b](\str) + \half\norms{\Atr^{1/2}(x-\str)}^2$. 
The equality~\eqref{eq:tr-gamma-pivot-outline} with $\lambda = \ltr$ and 
$\norm{x_t}\le 
\norm{\str}$ therefore implies
\begin{equation}\label{eq:tr-lin-subopt-bound}
\f(x_t) - \f(\str) \le \half\norm{\Atr^{1/2}(x_t - \str)}^2 + 
\ltr\norm{\str}(\norm{\str} - \norm{x_t}).
\end{equation}
When $\norm{y_t} \ge \norm{\str}$ we have $\norm{x_t} = \norm{\str}$ 
and the second term vanishes. When $\norm{y_t} < \norm{\str}$,
\begin{equation}\label{eq:tr-lin-norm-diff}
\norm{\str} - \norm{x_t} = \norm{\str} - \norm{y_t} - 
\frac{\norm{y_t}}{\norm{\str}}\cdot (\norm{\str} - \norm{y_t} )
= \norm{\str}\alpha^2 \le 4e^{-4t/\sqrt{\kappa(\Atr)}}\norm{\str}.
\end{equation}
We also have, 
\begin{align}\label{eq:tr-linear-quad}
  \lefteqn{\norm{\Atr^{1/2}(x_t - \str)} = 
    \norm{\left([1-\alpha]p(\Atr)\Atr-I\right)\Atr^{1/2}\str}} \nonumber \\
  & ~~ \le (1+|\alpha|)\norm{\left(p(\Atr)\Atr-I\right)\Atr^{1/2}\str} + 
  |\alpha| \norm{\Atr^{1/2}\str} %
  \le 6 \norm{\Atr^{1/2}\str}  e^{-2t/\sqrt{\kappa(\Atr)}},
\end{align}
where in the final transition we used our upper bounds on $\alpha$ 
and $\opnorm{p(\Atr)\Atr - I}$, as well as $|\alpha|\le 1$.
Substituting the bounds~\eqref{eq:tr-lin-norm-diff}
and~\eqref{eq:tr-linear-quad} into
inequality~\eqref{eq:tr-lin-subopt-bound}, we have
\begin{equation}\label{eq:tr-lin-time-bound-stronger}
\f(x_t) - \f(\str) \le \left(18 \str^T \Atr \str +  
4\ltr\norm{\str}^2\right) 
e^{-4t/\sqrt{\kappa(\Atr)}},
\end{equation}
and the final bound follows from recalling that 
$\f(0)-\f(\str) = \half \str^T \Atr \str + \frac{\ltr}{2}\norm{\str}^2$ 
and substituting $\kappa(\Atr) = (\lmax + \ltr)/(\lmin + \ltr)$. 
To conclude the proof we note 
that 
$(1-\alpha)p(\Atr) = (1-\alpha)p(A + \ltr I) = \tilde{p}(A)$ for some 
$\tilde{p} \in \mc{P}_t$, so that $x_t \in \mc{K}_t(A, b)$ and 
$\norm{x_t}\le R$, 
and therefore $\f(\itertr_t) \le \f(x_t)$.

\paragraph{Sublinear convergence}\label{sec:upper-sub-proof}

\newcommand{\Agam}{A_{0}}
\newcommand{\sgam}{\tilde{x}^\star_0}

Let $\Agam \defeq  A - \lmin I \succeq 0$ and apply 
Lemma~\ref{lem:agd} with $M=\Agam$, $v=b$ and $r=\norm{\str}$ to 
obtain $y_t \in \Krylov[t][\Agam, b] = \Krylov$ such that 
\begin{equation}\label{eq:tr-sublin-yt}
\norm{y_t}\le 
\norm{\str}
~\mbox{and}~
\f[\Agam,b](y_t) - \f[\Agam,b](\str) \le 
\f[\Agam,b](y_t) - \min_{\norm{x}\le\norm{\str}}\f[\Agam,b](x) \le 
\frac{4\opnorm{\Agam}\norm{\str}^2}{(t+1)^2}.
\end{equation}
If $\lmin \ge 0$, equality~\eqref{eq:tr-gamma-pivot-outline} with 
$\lambda=-\lmin$ along with~\eqref{eq:tr-sublin-yt} means we are done, 
recalling that $\opnorm{\Agam} = \lmax-\lmin$. 
For $\lmin< 0$, apply 
Lemma~\ref{lem:eigenvec-tight} with $M=\Agam$ and $v=b$ to obtain 
$z_t\in\Krylov$ such that
\begin{equation}\label{eq:tr-sublin-zt}
\norm{z_t} = 1
~~\mbox{and}~~
z_t^T \Agam z_t \le \frac{\opnorm{\Agam}}{16(t-\half)^2}
\log^2\left(4\frac{\norm{b}^2}{(u_{\min}^T b)^2}\right).
\end{equation}
We form the vector
\begin{equation*}
x_t = y_t + \alpha \cdot z_t\in\mc{K}_t(A,b),
\end{equation*}
and choose $\alpha$ to satisfy
\begin{equation*}
\norm{x_t} = \norm{\str}
~~\mbox{and}~~
\alpha \cdot z_t ^T (\Agam y_t  + b) = 
\alpha \cdot z_t ^T  \grad \f[\Agam,b](y_t) \le 0.
\end{equation*}
We may 
always choose such $\alpha$ because $\norm{y_t}\le\norm{\str}$ and 
therefore $\norm{y_t + \alpha z_t} = \norm{\str}$ has both a 
non-positive and a non-negative solution in $\alpha$. Moreover because 
$\norm{z_t}=1$ 
we have that $|\alpha| \le 2\norm{\str}$. 
The property 
$\alpha \cdot z_t ^T  \grad \f[\Agam,b](y_t) \le 0$ of our 
construction of $\alpha$ along with $\hess \f[\Agam,b] = \Agam$, 
gives us,
\begin{equation*}
\f[\Agam,b](x_t) = \f[\Agam,b](y_t) + \alpha \cdot z_t ^T  \grad 
\f[\Agam,b](y_t) + \frac{\alpha^2}{2} z_t^T \Agam z_t
\le \f[\Agam,b](y_t) + \frac{\alpha^2}{2} z_t^T \Agam z_t.
\end{equation*}
Substituting this bound along with $\norm{x_t}=\norm{\str}$ and 
$\alpha^2 \le 4\norm{\str}^2$ into~\eqref{eq:tr-gamma-pivot-outline} 
with $\lambda=-\lmin$ gives
\begin{equation*}
\f(x_t) - \f(\str) \le \f[\Agam,b](y_t) - \f[\Agam,b](\str) + 
2\norm{\str}^2 
z_t^T \Agam z_t.
\end{equation*}
Substituting in the bounds~\eqref{eq:tr-sublin-yt} 
and~\eqref{eq:tr-sublin-zt} concludes the proof for the case $\lmin < 
0$.

\subsection{Randomizing away the hard case}\label{sec:upper-rand}

\newcommand{\jitertr}{\hat{s}^{\mathsf{tr}}}
\newcommand{\pitertr}{\tilde{s}^{\mathsf{tr}}}
\newcommand{\pb}{\tilde{b}}
\newcommand{\pstr}{\tilde{x}^{\star}_{\mathsf{tr}}}
\newcommand{\pltr}{\tilde{\lambda}^{\star}}
\newcommand{\pAtr}{A_{\pltr}}

Krylov subspace solutions may fail to converge to global solution when both
$\ltr = -\lmin$ and $u_{\min}^T b = 0$, the so-called hard
case~\cite{ConnGoTo00,NocedalWr06}. Yet as with eigenvector
methods~\cite{KuczynskiWo92,GolubVa89}, simple randomization
approaches allow us to handle the hard case with high probability, at the
modest cost of introducing to the error bounds a logarithmic dependence 
on $d$. Here we describe two such approaches.

In the first approach, we draw a spherically symmetric random vector $v$, 
and consider the \emph{joint Krylov subspace}
\begin{equation*}
	\Krylov[2t][A, \{b, v\}] \defeq \mathrm{span}\{b, Ab, \ldots, A^{t-1}b, 
	v, Av, \ldots, A^{t-1}v\}.
\end{equation*}
The trust-region and cubic-regularized problems~\eqref{eq:problems} can 
be solved efficiently in $\Krylov[2t][A, \{b, v\}]$ using the \emph{block 
Lanczos} method~\cite{CullumDo74,Golub77}; we survey this technique in 
Section~\ref{sub:block-lanczos} in the 
supplement. The analysis in the previous section immediately implies the 
following convergence guarantee.
\begin{restatable}{corollary}{corRandJoin}\label{cor:tr-rand-joint}
	Let $v$ be uniformly distributed on the unit sphere in $\R^d$, and 
	$$\jitertr_t \in \argmin_{x\in\Krylov[\floor{t/2}][A,\{b,v\}], \norm{x}\le 
	R} \f[A,b](x).$$ For any $\delta > 0$, 
	\begin{equation}\label{eq:tr-joint-sublin-time}
	\f(\jitertr_t) - \f(\str) 
	\le \frac{(\lmax - 
		\lmin)R^2}{(t-1)^2} 
	\left[ 
	16 + 2\cdot\I_{\{\lmin < 0\}}
	\log^2\left(\frac{2\sqrt{d}}{ \delta}\right)
	\right]
	\end{equation}
	with probability at least $1-\delta$ with respect to the random choice 
of 	$v$. 
\end{restatable}
\begin{proof}
	In the preceding proof of sublinear convergence, 
	apply Lemma~\ref{lem:agd} on $\Krylov[\floor{t/2}][A,b]$ and 
	Lemma~\ref{lem:eigenvec-tight} on $\Krylov[\floor{t/2}][A,v]$; the 
	constructed solution is in $\Krylov[\floor{t/2}][A,\{b,v\}]$. To bound 
	$|u_{\min}^T v|^2/\|v\|^2$, note that its 
	distribution is 
	$\textrm{Beta}(\frac{1}{2}, \frac{d-1}{2})$ and therefore $|u_{\min}^T 
	v|^2/\|v\|^2 \ge \delta^2 / d$ with probability greater than $1-\delta$ 
	(cf.~\cite[Lemma 4.6]{CarmonDu16}).
\end{proof}

Corollary~\ref{cor:tr-rand-joint} implies we can solve the trust-region 
problem to $\epsilon$ accuracy in roughly 
$\epsilon^{-1/2}\log d$ matrix-vector products, even in the 
hard case. The main drawback of this randomization approach is that 
half the matrix-vector products are expended on the random vector; when 
the problem is well-conditioned or when $|u_{\min}^T b|/\norms{b}$ is not 
extremely small, using the standard subspace solution is nearly twice as 
fast.

The second approach follows the proposal~\cite{CarmonDu16} to 
construct a perturbed version of the linear term $b$, denoted $\pb$, and 
solve the problem instance
$(A,\pb, R)$ in the Krylov subspace $\Krylov[t][A,\pb]$.
%
%
%
%
%
%

\begin{restatable}{corollary}{corRand}\label{cor:tr-rand}
	Let $v$ be uniformly distributed on the unit sphere in $\R^d$, let 
	$\sigma 
	> 0$ and let
	\begin{equation*}
	\pb = b  + \sigma \cdot v.
	\end{equation*}
	Let
	$\pitertr_t \in \argmin_{x\in\Krylov[t][A,\pb], \norm{x}\le R} 
	\f[A,\pb](x) \defeq
	\half x^T A x + \pb^T x$. For any $\delta > 0$, 
		\begin{equation}\label{eq:tr-pert-sublin-time}
	\f(\pitertr_t) - \f(\str) \le \frac{(\lmax - 
		\lmin)R^2}{(t-\half)^2} 
	\left[ 
	4 + \frac{\I_{\{\lmin < 0\}}}{2}
	\log^2\left(\frac{2\norms{\pb}\sqrt{d}}{\sigma \delta}\right)
	\right] + 2\sigma R
	\end{equation}
	with probability at least $1-\delta$ with respect to the random choice of 
	$v$. 
\end{restatable}
\noindent
See section~\ref{sec:tr-rand-proof} in the supplement for a short proof, 
which consists of arguing that $\f$ and $\f[A,\pb]$ deviate by at most 
$\sigma R$ at any feasible point, and applying a probabilistic lower bound 
on $|u_{\min}^T \pb|$. For any desired accuracy $\epsilon$, using 
Corollary~\ref{cor:tr-rand} with $\sigma = \epsilon/(4R)$ shows we can 
achieve this accuracy, with constant probability, in a number of Lanczos 
iterations that scales as 
$\epsilon^{-1/2}\log(d/\epsilon^2)$. Compared to the 
first approach, this rate of convergence is asymptotically slightly slower (by 
a factor of $\log{\frac{1}{\epsilon}}$), and moreover requires us to decide 
on a desired level of accuracy in advance. 
However, the second approach avoids the 2x slowdown that the first 
approach exhibits on easier problem instances. In Section~\ref{sec:exp} we  
compare the two approaches empirically.

We remark that the linear convergence guarantee~\eqref{eq:tr-lin-time} 
continues to hold for both randomization approaches. For the second 
approach, this is due to the fact that small perturbations to $b$ do not 
drastically change the condition number, as shown in~\cite{CarmonDu16}. 
However, this also means that we cannot expect a good condition number 
when perturbing $b$ in the hard case. Nevertheless, we believe it is 
possible to show that, with  
randomization, Krylov subspace methods exhibit linear 
convergence even in the hard case, where the condition number is replaced 
by the normalized eigen-gap $(\lmax-\lmin)/(\lambda_2 - \lmin)$, with 
$\lambda_2$  the smallest eigenvalue of $A$ larger than $\lmin$.

%% file: cubic.tex
\section{The cubic-regularized problem}\label{sec:upper-cr}

We now consider the cubic-regularized problem
\begin{equation*}%
\minimize_{x\in\R^{d}}~\fcu\left(x\right)\defeq
\f(x) +\frac{\rho}{3} \norm{x}^3 = 
\half x^{\T}Ax+b^{\T}x+\frac{\rho}{3} \norm{x}^3.
\end{equation*}
Any global minimizer of $\fcu$, denoted $\scu$, admits the 
characterization~\cite[Theorem 3.1]{CartisGoTo11}
\begin{equation}
\grad \fcu(\scu) = (A+\rs I)\, \scu+b =0 
~~ \mbox{and} ~~ \rho \norm{\scu} \ge -\lmin.
\label{eq:cr-optimality}
\end{equation}
Comparing this characterization to its counterpart~\eqref{eq:tr-optimality} 
for the trust-region problem, we see that any instance $(A,b,\rho)$ of cubic 
regularization has an \emph{equivalent trust-region instance} $(A,b,R)$, 
with $R=\norm{\scu}$. Theses instances are equivalent in that they have 
the same set of global minimizers. Evidently, the equivalent trust-region 
instance has optimal Lagrange multiplier $\ltr = \rs$. Moreover, at any 
trust-region feasible point $x$ (satisfying $\norm{x}\le 
R=\norm{\scu}=\norm{\str}$), the cubic-regularization optimality gap is 
smaller than its trust-region equivalent,
\begin{equation*}
\fcu(x)-\fcu(\scu) = \f(x) - \f(\str) 
+ \frac{\rho}{3}\big( \norm{x}^3 - \norms{\str}^3 \big)
\le \f(x) - \f(\str).
\end{equation*}
Letting $\itercu_t$ denote the minimizer of $\fcu$ in $\Krylov$ and letting 
$\itertr_t$ denote the Krylov subspace solution of the equivalent 
trust-region 
problem, we conclude that
\begin{equation}\label{eq:cr-tr-opt-gap-bound}
\fcu(\itercu_t) - \fcu(\scu) \le \fcu(\itertr_t) - \fcu(\scu) \le
\f(\itertr_t) - \f(\str);
\end{equation}
cubic regularization Krylov subspace solutions always have a 
\emph{smaller optimality gap} than their 
trust-region equivalents. The guarantees of Theorem~\ref{thm:tr} therefore 
apply to $\fcu(\itercu_t) - \fcu(\scu)$ as well, and we arrive at the 
following
\begin{corollary}\label{cor:cu}
	For every $t>0$, 
	\begin{equation}\label{eq:cr-lin-time}
	\fcu(\itercu_t) - \fcu(\scu) \le 36\left[\fcu(0) - \fcu(\scu) \right]
	\exp\left\{
	-4t\sqrt{\frac{\lmin + \rs}{\lmax +\rs}}
	\right\},
	\end{equation}
	and
	\begin{equation}\label{eq:cr-sublin-time}
	\fcu(\itercu_t) - \fcu(\scu) \le \frac{(\lmax - 
		\lmin)\norm{\scu}^2}{(t-\half)^2} 
	\left[ 
	4 + \frac{\I_{\{\lmin < 0\}}}{8}
	\log^2\left(\frac{4\norm{b}^2}{(u_{\min}^T b)^2}\right)
	\right].
	\end{equation}
\end{corollary}
\begin{proof}
  Use the slightly stronger bound~\eqref{eq:tr-lin-time-bound-stronger}
  derived in the proof of Theorem~\ref{thm:tr} with the
  inequality $18 \str^T \Atr
  \str + 4\ltr\norm{\str}^2 \le 36[ \half\scu^T A \scu + \frac{1}{6}\rho 
  \norm{\scu}^3] =
  36[\fcu(0)-\fcu(\scu)]$.
\end{proof}

Here too it is possible to randomly perturb $b$ and obtain a guarantee for 
cubic regularization that applies in the hard case. In~\cite{CarmonDu16} 
we carry out such analysis for gradient descent, and show that 
perturbations to $b$ with norm $\sigma$ can increase $\norms{\scu}^2$ 
by 
at most $2\sigma/\rho$ \cite[Lemma 4.6]{CarmonDu16}. Thus the 
cubic-regularization equivalent of Corollary~\ref{cor:tr-rand} amounts to 
replacing $R^2$ with $\norm{\scu}^2+2\sigma/\rho$ 
in~\eqref{eq:tr-pert-sublin-time}.

We note briefly---without giving a full analysis---that
Corollary~\ref{cor:cu} shows that the practically successful Adaptive
Regularization using Cubics (ARC) method~\cite{CartisGoTo11} can find
$\epsilon$-stationary points in roughly $\epsilon^{-7/4}$ Hessian-vector
product operations (with proper randomization and subproblem stopping
criteria). Researchers have given such guarantees for a number of algorithms
that are mainly theoretical~\cite{AgarwalAlBuHaMa17,CarmonDuHiSi18}, as 
well as variants of accelerated gradient 
descent~\cite{CarmonDuHiSi17,JinNeJo17}, which while more practical still 
require careful parameter tuning. In contrast, ARC requires very little tuning 
and it is encouraging that it may also exhibit the enhanced 
Hessian-vector product complexity $\epsilon^{-7/4}$, which is at least 
near-optimal~\cite{CarmonDuHiSi17lii}.

%% file: lower.tex
\arxiv{\vspace{-6pt}}
\section{Lower bounds}\label{sec:lower}
\arxiv{\vspace{-6pt}}

\newcommand{\probparams}{\mathfrak{L}}

We now show that the guarantees in Theorem~\ref{thm:tr} and
Corollary~\ref{cor:cu} are tight up to numerical constants for adversarially
constructed problems. We state the result for the cubic-regularization
problem; corresponding lower bounds for the trust-region problem are
immediate from the optimality gap 
relation~\eqref{eq:cr-tr-opt-gap-bound}.\footnote{To obtain the correct 
prefactor in the trust-region equivalent of lower bound~\eqref{eq:cr-lb-lin} 
we may use the fact that $\fcu(0) - \fcu(\scu) = \frac{1}{2}b^T \Atr^{-1} b 
+ 
\frac{\rho}{6}\norm{\scu}^3 \ge 
\frac{1}{3}(\frac{1}{2} b^T \Atr^{-1} b + \frac{\ltr}{2}R^2)=
\frac{1}{3}(\f(0)-\f(\str))$.}

To state the result, we require a bit more notation. Let $\probparams$ map 
cubic-regularization problem instances of the form $(A,b,\rho)$ to the 
quadruple $(\lmin, \lmax, \ltr, \Delta)=\probparams(A,b,\rho)$ such that 
 $\lmin,\lmax$ are 
the extremal eigenvalues of $A$ and the solution $\scu = \argmin_x 
\fcu(x)$ satisfies $\rho\norm{\scu} = \ltr$, and $\fcu(0) - \fcu(\scu) = 
\Delta$. Similarly let $\probparams'$ map an instance $(A,b,\rho)$ to the 
quadruple $(\lmin, \lmax, \tau, R)$ where now $\norm{\scu}=R$ and 
$\norm{b}/|u_{\min}^T b|=\tau$, with $u_{\min}$ an eigenvector of $A$ 
corresponding to eigenvalue $\lmin$. 

With this notation in hand, we state our lower bounds. (See supplemental 
section~\ref{sec:lb-proof} for a proof.)

\begin{theorem}\label{thm:lb}
  Let $d, t \in \N$ with $t < d$
  and
  $\lmin,\lmax,\ltr,\Delta$
  be such that $\lmin \le \lmax$, $\ltr
  > (-\lmin)_+$, and $\Delta > 0$.
  There exists $(A, b, \rho)$ such that $\probparams(A,b,\rho) = (\lmin, 
  \lmax, \ltr, \Delta)$ and  for all $s\in\Krylov$,
  \begin{equation}\label{eq:cr-lb-lin}
    \fcu(s) - \fcu(\scu) > 	 
    \frac{1}{K}
    \left[\fcu(0) - \fcu(\scu) \right]
    \exp\left\{
    -\frac{4t}{\sqrt{\kappa}-1}
    \right\},
  \end{equation}
  where $K = 1+\frac{\ltr}{3(\ltr + \lmin)} $ and $\kappa = \frac{\ltr + 
  \lmax}{\ltr +\lmin}$. Alternatively, for any $\tau\ge 1$ and $R>0$, there 
  exists $(A, b, \rho)$ such that $\probparams'(A,b,\rho) = (\lmin, \lmax, 
  \tau, R)$ and for $s\in\Krylov$,
  \begin{equation}\label{eq:cr-lb-sub-ncvx}
    \fcu(s) - \fcu(\scu) > 
    \min\left\{
    (\lmax)_- - \lmin,
    \frac{\lmax-\lmin}{16(t-\half)^2}
    \log^2\left(\frac{\norm{b}^2}{( u_{\min}^T b)^2}\right)
    \right\}\frac{\norm{\scu}^2}{32},
  \end{equation}
  and %
  \begin{equation}\label{eq:cr-lb-sub-cvx}
    \fcu(s) - \fcu(\scu) > 
    \frac{(\lmax - \lmin)\norm{\scu}^2}{16(t+\half)^2}.
  \end{equation}
\end{theorem}

The lower bounds~\eqref{eq:cr-lb-lin} matches
the linear convergence guarantee~\eqref{eq:cr-lin-time} to within a
numerical constant, as we may choose $\lmax,\lmin$ and $\ltr$ so that 
$\kappa$ is arbitrary and $K < 2$.  Similarly, lower 
bounds~\eqref{eq:cr-lb-sub-ncvx}
and~\eqref{eq:cr-lb-sub-cvx} match the sublinear convergence
rate~\eqref{eq:cr-sublin-time} for $\lmin < 0$ and $\lmin \ge 0$
respectively.  Our proof flows naturally from 
minimax characterizations of uniform polynomial approximations
(Lemmas~\ref{lem:cheby-first} and~\ref{lem:cheby-second} in the supplement),
which also play a crucial role in proving our upper bounds.

One consequence of the lower bound~\eqref{eq:cr-lb-lin} is the 
existence of extremely badly conditioned instances, say with $\kappa = 
(100d)^2$ and $K=3/2$, such that in the first $d-1$ iterations it is 
impossible to decrease the initial error by more than a factor of 2 (the initial 
error may be chosen arbitrarily large as well). However, since these 
instances have finite condition number we have $\scu\in\Krylov[d]$, and so 
the error supposedly drops to 0 at the $d$th iteration. This seeming 
discontinuity 
stems from the fact that in this case $\scu$ depends on the Lanczos basis 
of $\Krylov[d]$ through a very badly conditioned linear system and cannot 
be recovered with finite-precision arithmetic. Indeed, 
running Krylov subspace methods for $d$ iterations with inexact  
arithmetic often results in solutions that are very far from exact, while 
guarantees of the form~\eqref{eq:cr-lin-time} are more robust to 
roundoff errors~\cite{MuscoMuSi17,DruskinKn91,TrefethenBa97}.

While we state the lower bounds in Theorem~\ref{thm:lb} for points in the
Krylov subspace $\Krylov$, a classical ``resisting oracle'' construction due
to~\citet[Chapter 7.2]{NemirovskiYu83} (see
also~\cite[\S10.2.3]{Nemirovski94}) shows that (for $d>2t$) these lower
bounds hold also for \emph{any deterministic method} that accesses $A$ only
through matrix-vector products, and computes a single matrix-vector product
per iteration. The randomization we employ in 
Corollaries~\ref{cor:tr-rand-joint} and~\ref{cor:tr-rand}
breaks the lower bound~\eqref{eq:cr-lb-sub-ncvx} when $\lmin < 0$ and
$\norm{b}/|u_{\min}^T b|$ is very large, so there is some substantial power
from randomization in this case. However, \citet{Simchowitz18} recently 
showed that randomization 
cannot break the lower bounds for convex quadratics ($\lmin\ge0$ and 
$\rho=0)$.
%

%% file: experiment.tex
\arxiv{\vspace{-6pt}}
\section{Numerical experiments}\label{sec:exp}
\arxiv{\vspace{-3pt}}
To see whether our analysis applies to non-worst case problem instances, 
we generate 5,000 random cubic-regularization problems with 
$d=10^{6}$ and controlled condition number 
$\kappa=(\lmax+\rs)/(\lmin+\rs)$ (see Section~\ref{sec:exp-details} in 
the supplement for more details). We repeat the experiment three 
times with different values of $\kappa$ and summarize the results in 
Figure~\ref{fig:exp}a.  
As seen in the figure, about 20 Lanczos iterations suffice to solve even the 
worst-conditioned instances to about $10\%$ accuracy, and 100 
iterations give accuracy better than $1\%$. Moreover, for $t \gtrapprox 
\sqrt{\kappa}$, the approximation error decays exponentially with precisely 
the rate $4/\sqrt{\kappa}$ predicted by our analysis, for almost all the 
generated problems. For $t \ll \sqrt{\kappa}$, the error decays 
approximately as $t^{-2}$. We conclude that the rates characterized by  
Theorem~\ref{thm:tr} are relevant beyond the worst case.

We conduct an additional experiment to test the effect of randomization for 
``hard case'' instances, where $\kappa = \infty$.
We generate such problem instances 
(see details in Section~\ref{sec:exp-details}), and 
compare the joint subspace randomization scheme 
(Corollary~\ref{cor:tr-rand-joint}) to the perturbation scheme 
(Corollary~\ref{cor:tr-rand}) with different perturbation magnitudes 
$\sigma$; the results are shown in Figure~\ref{fig:exp}b. For any fixed 
target accuracy, some choices of $\sigma$ 
yield faster convergence than the joint subspace scheme. However, for any 
fixed $\sigma$ optimization eventually hits a noise floor due to the 
perturbation, while the joint subspace scheme continues to improve. 
Choosing $\sigma$ requires striking a balance: if too large  
the noise floor is high and might even be worse 
than no perturbation at all; if too small, escaping the unperturbed error 
level will take too long, and the method might falsely declare convergence. 
A 
practical heuristic for safely choosing 
$\sigma$ 
is 
an interesting topic for future research.

\begin{figure}
	    \centering
	\hspace{-0.75cm}
	\begin{minipage}[t]{0.72\textwidth}
		\centering
		\includegraphics[height=5.025cm]{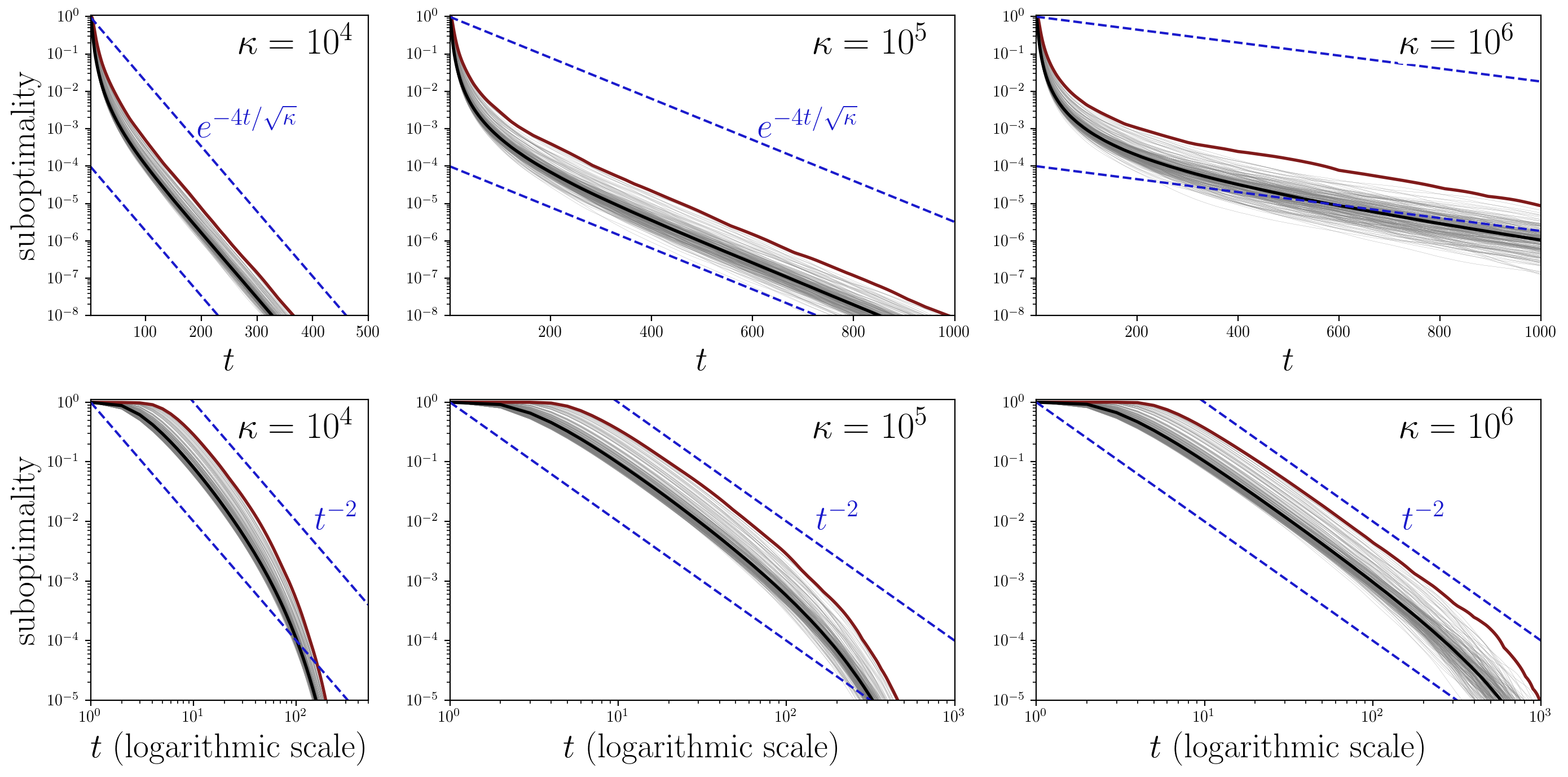}
		\\
		\footnotesize\textbf{(a)}
	\end{minipage}
	\begin{minipage}[t]{0.27\textwidth}
		\centering
		\includegraphics[height=5.1cm]{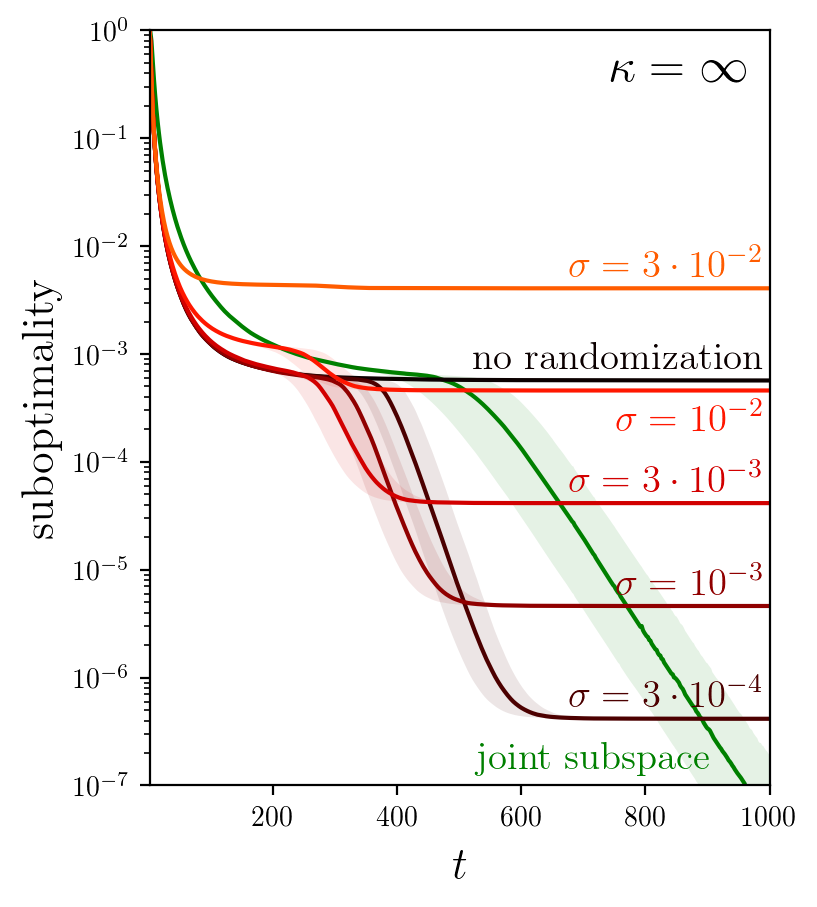}
		\\
		\nips{~~~~~~~~~~~~~~~~}
		\arxiv{~~~~~}
		\footnotesize\textbf{(b)}
	\end{minipage}
	\caption{\label{fig:exp}Optimality gap of 
	Krylov subspace solutions on random 
	cubic-regularization problems, versus subspace dimension 
	$t$. \textbf{(a)} Columns show ensembles with different 
	condition numbers 
	$\kappa$, and rows differ by scaling of $t$. Thin lines indicate 
	results for individual instances, and bold lines indicate ensemble median 
	and maximum suboptimality. 
	\textbf{(b)} Each line represents median suboptimality, and shaded 
	regions represent inter-quartile range.
	Different lines 
	correspond to different randomization settings.
}
\end{figure}

%% file: lanczos.tex
\section{Computing Krylov subspace solutions}\label{sec:lanczos}
Generic instances of the trust-region and cubic-regularized problems can 
be globally optimized by solving the one-dimensional 
equations
\begin{equation}\label{eq:tr-lambda-search}
\norm{A_{\lambda}^{-1}b} = R~,~
\lambda  >  \max\{-\lmin, 0\}.
\end{equation}
and
\begin{equation}\label{eq:cr-lambda-search}
\norm{A_\lambda^{-1}b} = \lambda/\rho
~~,~~
\lambda \ge -\lmin,
\end{equation}
respectively. However, when $d$ is very large, even a single exact 
evaluation of $\norms{A_\lambda^{-1}b}$ (which requires a direct linear 
system solution)  can become prohibitively expensive.

In this case, a general approach to obtaining approximate solutions is to 
constrain the domain to a linear subspace $\mc{Q}_t \subset \R^d$ of 
dimension $t \ll d$. Let $Q_t \in \R^{d\times t}$ be an orthogonal basis for 
$\mc{Q}_t$ ($Q_t^T Q_t = I$). Finding the global minimizer in $\mc{Q}_t$ 
is equivalent to re-parameterizing $x$ as $x=Q_t \tilde{x}$ and solving for 
$\tilde{x}\in\R^t$, which is also equivalent to solving a $t$-dimensional 
problem instance with $\tilde{A} = Q_t^T A Q_t$ and $\tilde{b} = Q_t^T b$. 
For sufficiently large $d$, the time to solve such problems will be 
dominated by the $t$ matrix-vector products required to construct 
$\tilde{A}$.

In this paper we focus on the choice $\mc{Q}_t = \Krylov$ the Krylov 
subspace of order $t$. This choice offers a significant efficiency boost: we 
can efficiently construct a basis $Q_t$ for which $Q_t^T A Q_t$ is 
tridiagonal, using the Lanczos process, which consists of the 
following recursion, starting with $q_1 = b/\norm{b}, q_0 = 0$,
\begin{equation*}
\alpha_t = q_t^T A q_t
~,~
q'_{t+1} = A q_t - \alpha_t q_t - \beta_t q_{t-1}
~,~
\beta_{t+1} = \norm{q'_{t+1}}
~,~
q_{t+1} = q'_{t+1}/\norm{q'_{t+1}}.
\end{equation*}
The vectors $q_1, \ldots, q_t$ give the columns of $Q_t$ while $\alpha_1, 
\ldots, \alpha_t$ and $\beta_2, \ldots, \beta_t$ respectively give the 
diagonal and off-diagonal elements of the symmetric tridiagonal matrix 
$\tilde{A} = Q_t^T A Q_t$; this makes solving 
equations~\eqref{eq:tr-lambda-search} and~\eqref{eq:cr-lambda-search} 
easy. 
One straightforward approach is to directly compute the factorization 
$\tilde{A}$, which for a symmetric tridiagonal matrix of size $t$ takes 
$O(t\log t)$ time~\cite{CoakleyRo13}. A more efficient approach---and the 
one used in practice---is to iteratively solve systems of the form  
$\tilde{A}_{\lambda}x = z$ and update $\lambda$ using Newton 
steps~\cite{ConnGoTo00,CartisGoTo11}. Every tridiagonal system solution 
can be done in time $O(t)$, and the Newton steps are shown 
in~\cite{ConnGoTo00,CartisGoTo11} to be globally linearly convergent, with 
local quadratic convergence. In our experience less than 20 Newton steps 
generally suffice to reach machine precision, and so the computational cost 
is essentially linear in $t$. 
%
%
%
%
%
It is also possible to avoid keeping $Q_t$ in memory (when $t\cdot d$ 
storage is too demanding) by running the Lanczos process twice, once for 
evaluating $\tilde{x}$ and again to obtain $x=Q_t \tilde{x}$.

The Lanczos process produces the same result as Gram-Schmidt 
orthonormalization of the vectors $\left[b, Ab, \ldots, A^{t-1}b\right]$ but 
uses the 
special structure of that matrix to avoid computing inner products that are 
known in advance to be zero. When run for many iterations, the Lanczos 
process has well-documented numerical stability 
issues~\cite{TrefethenBa97}. However, in 
our setting we usually seek low to moderate accuracy solutions and will 
usually stop at $t < 100$, for which Lanczos is reasonably stable with 
floating point arithmetic even when $d$ is quite large. The application of 
the Lanczos process---which is typically used for eigenvector 
computation---in the context of regularized quadratic optimization is 
sometimes referred to as the generalized Lanczos 
process~\cite{GouldLuRoTo99}.

\subsection{Computing joint Krylov subspace 
solutions}\label{sub:block-lanczos}

To solve equations~\eqref{eq:tr-lambda-search} 
and~\eqref{eq:cr-lambda-search} in subspaces of the form
\begin{equation*}
\Krylov[mt][A, \{v_1, \ldots, v_m\}] \defeq \mathrm{span}\{A^j 
v_i\}_{i\in\{1,\ldots, m\},j\in\{0,\ldots,t-1\}}
\end{equation*}
we may use the block Lanczos method~\cite{CullumDo74,Golub77}, a 
natural generalization of 
the Lanczos method that creates an orthonormal basis for the subspace 
$\Krylov[mt][A, \{v_1, \ldots, v_m\}]$ in which $A$ has a block tridiagonal 
form. Overloading the notation defined above so that now $q_t \in 
\R^{d\times m}$ and $\alpha_t, \beta_t \in \R^{m\times m}$, the block 
Lanczos recursion is given by,
\begin{equation*}
\alpha_t = q_t^T A q_t
~,~
q'_{t+1} = A q_t - q_t \alpha_t -  q_{t-1} \beta_t^T
~,~
(q_{t+1}, \beta_{t+1}) = \mathrm{QR}(q'_{t+1}).
\end{equation*}
where $\mathrm{QR}$ stands for the QR decomposition (i.e.\ if $(q, 
\beta) = \mathrm{QR}(a)$ then $q$ is orthogonal, $\beta$ is upper 
diagonal and $a 
= q\cdot \beta$), and the initial conditions are that $q_1$ is an 
orthonormalized version of $[v_1, \ldots, v_m]$ and $q_0=0$. The matrix 
$\tilde{A} = Q_t^T A Q_t$ is now block tridiagonal, with the diagonal and 
sub-diagonal blocks given by $\{\alpha_i\}_{i\in\{1,\ldots,t\}}$ and 
$\{\beta_i\}_{i\in\{2,\ldots,t\}}$ respectively. Since the $\beta$ matrices 
are upper diagonal, $\tilde{A}$ is a symmetric banded matrix with $m$ 
non-zeros sub-diagonal bands. Such matrix admits fast Cholesky 
decomposition (in time linear in $m^2 t$), and consequently the Newton 
method described above is still efficient.

%% file: polys.tex
\section{Polynomial approximation results}\label{sec:polys}

In this section we state (and prove for ease of reference) two classical 
results on uniform polynomial approximation 
(cf.~\cite{KuczynskiWo92,Nemirovski94}) that stand at the core of the 
technical development in this work.

\begin{lemma}\label{lem:cheby-first}
	Let $n \ge 1$ and $0< \alpha \le \beta$, and let $\kappa = 
	\beta/\alpha$. Then
	\begin{equation*}
	\min_{p \in \Polys} \max_{x\in [\alpha,\beta]} | 1 - x p(x) | = 
		\minmaxT \defeq
		2 \left(  \left(\frac{\sqrt{\kappa}+1}{\sqrt{\kappa}-1}\right)^n 
		+ \left(\frac{\sqrt{\kappa}-1}{\sqrt{\kappa}+1}\right)^n \right)^{-1}
	\end{equation*}
	and
	\begin{equation*}
	2\left(e^{2n/(\sqrt{\kappa}-1)}+1\right)^{-1} \le \minmaxT \le 
	2e^{-2n/\sqrt{\kappa}}.
	\end{equation*}
	Moreover, there exist $x_0, x_1, \ldots, x_{n} \in [\alpha,\beta]$ and 
	probability distribution $\pi_0, \pi_1, \ldots \pi_{n}$ such that
	\begin{equation*}
	\min_{p \in \Polys} \sum_{k=0}^{n} \pi_k (1-x_k p(x_k))^2 = 
	[\minmaxT] ^2.
	\end{equation*}
\end{lemma}

\begin{proof}
	Let
	\begin{equation*}
	T_n(x) = \begin{cases}
	\cos(n\arccos(x)) & |x| \le 1 \\
	\half \left(  (x + \sqrt{x^2-1})^n + (x - \sqrt{x^2-1})^n \right) & |x| 
	\ge 1
	\end{cases}
	\end{equation*}
	denote the order $n$ Chebyshev polynomial of the first kind. We 
	claim that $p^\star\in\Polys$ that solves the minimax problem $\min_{p 
	\in 
	\Polys} \max_{x\in [\alpha,\beta]} | 1 - x p(x) |$ is given by
	\begin{equation*}
	1-xp^\star(x) = \minmaxT\cdot
	 T_n \left( 
	\frac{\kappa+1-2x/\alpha}
	{\kappa-1} \right),
	\end{equation*}
	where $\minmaxT = \left[ T_n \left( \frac{\kappa+1}{\kappa-1} 
	\right)\right]^{-1}$ guarantees that the RHS has value 1 at $x=0$ and 
	therefore $p^\star$ is well defined. Since clearly $|T_n(y)|\le 1$ for every 
	$y\in[-1,1]$, we have that $\max_{x\in [\alpha,\beta]} | 1 - x p^\star(x) 
	| = \minmaxT$. 
	
	We argue that $p^\star$ is optimal using the classical alternating signs 
	argument, sometimes also referred to as Chebyshev's theorem. 
	First, note that $T_n(y)$ has $n+1$ extrema in $[-1,1]$ (at 
	$y_k=\cos(k\pi/n)$ for $k=0,\ldots,n$) and that their values alternate 
	between $-1$ and $1$ (\ie $T_n(y_k)=(-1)^k$). Therefore, there exist 
	$n+1$ distinct points $x_0, x_1, \ldots, x_{n} \in [\alpha,\beta]$ for 
	which $1-x_i p^\star(x_k) = (-1)^{k}\minmaxT$. Let $q\in\Polys$ 
	satisfy $\max_{x\in [\alpha,\beta]} | 1 - x q(x) 
	| \le \minmaxT$. Then, 
	\[p^\star(x_k) - q(x_k) = \frac{[1-x_k q(x_k)] 
	-[1-x_k p^\star(x_k)]}{x_k} \]
	must be non-positive for 
	even $k$ and non-negative for odd $k$, and therefore $p^\star- q$ 
	must 
	have at least $n$ roots  
	in $[\alpha,\beta]$. However, $p^\star- q$ is a polynomial of degree at 
	most $n-1$ and can have $n$ roots only if it is identically 0, so we have 
	that $q=p^\star$, proving that  $p^\star$ is the unique 
	solution of the minimax problem. 
	
	To see the upper and lower bounds on $\minmaxT$, note that 
	$\minmaxT 
	= 1/\cosh(n \log( 1 + \frac{2}{\sqrt{\kappa}-1}))$, that $\half e^{|y|} \le 
	\cosh(y) \le \half(e^{|y|} +1)$, and that 
	\begin{equation*}
	\frac{2}{z} \le \log\left( 1 + \frac{2}{z-1}\right) \le 
	\frac{2}{z-1}
	\end{equation*}
	for all $z>1$, 
	where the lower bound above can seen by comparing derivatives.
	
	To see the final part of the lemma, let $x_0, x_1, \ldots, x_{n} \in 
	[\alpha,\beta]$ be the points constructed in the optimality argument 
	above, and note that this argument continues to hold if the inner 
	maximization is restricted to these points. Therefore,
	\begin{equation*}
	\min_{p \in \Polys}\max_{0\le k \le n} (1-x_k p(x_k))^2
	= \left[ \min_{p \in \Polys}\max_{0\le k \le n} |1-x_k p(x_k)|  \right]^2 
	= [\minmaxT]^2.
	\end{equation*}
	Letting $\Delta_{n+1}$ denote the probability simplex with $n+1$ 
	variables, we may write
	\begin{equation*}
	\max_{0\le k \le n} (1-x_k p(x_k))^2 = 
	\max_{\mu\in\Delta_{n+1}} 
	\sum_{k=0}^{n} \mu_k (1-x_k p(x_k))^2.
	\end{equation*}
	Finally, noting that the objective $\sum_{k=0}^{n} \mu_k (1-x_k 
	p(x_k))^2$ is linear (and hence concave) in $\mu$ and convex in (the 
	coefficients of) $p$, we 
	may use Von-Neumann's lemma and swap the $\min$ and $\max$ 
	above, writing 
	\begin{equation*}
	\max_{\mu\in\Delta_{n+1}} \min_{p \in \Polys} \sum_{k=0}^{n} \mu_k 
	(1-x_k p(x_k))^2
	=
	\min_{p \in \Polys}\max_{\mu\in\Delta_{n+1}}  \sum_{k=0}^{n} \mu_k 
	(1-x_k p(x_k))^2
	= [\minmaxT]^2.
	\end{equation*}
	Letting $\pi$ denote the distribution attaining the outer maximum, we 
	get the desired result. We remark in passing that $\pi$ may be 
	constructed explicitly using the orthogonality principle of least squares 
	estimation and orthogonality relations of Chebyshev polynomials. 
\end{proof}

\begin{lemma}\label{lem:cheby-second}
	Let $n \ge 1$ and $0< \alpha \le \beta$, let $\kappa = 
	\beta/\alpha$ and define $w(x)\defeq \sqrt{x-\alpha}$.  
	Then
	\begin{equation*}
	\min_{p \in \Polys} \max_{x\in [\alpha,\beta]} 
	{w(x)}| 1 - x 	p(x) 
	| = 
	\minmaxU \defeq
	2\sqrt{\alpha} \left(  
	\left(\frac{\sqrt{\kappa}+1}{\sqrt{\kappa}-1}\right)^{n + \half} 
	- \left(\frac{\sqrt{\kappa}-1}{\sqrt{\kappa}+1}\right)^{n + \half}  
	\right)^{-1}
	\end{equation*}
	and
	\begin{equation*}
	2\sqrt{\alpha} 
	\left(e^{2(2n+1)/(\sqrt{\kappa}-1)} - 1\right)^{-\half}
	\le \minmaxU \le 
	2\sqrt{\alpha} 
	\left(e^{2(2n+1)/\sqrt{\kappa}} - 2\right)^{-\half}.
	\end{equation*}
	Moreover, there exist $x_0, x_1, \ldots, x_{n} \in [\alpha,\beta]$ and 
	probability distribution $\pi_0, \pi_1, \ldots \pi_{n}$ such that
	\begin{equation*}
	\min_{p \in \Polys} \sum_{k=0}^{n} \pi_k w^2(x_k)(1-x_k p(x_k))^2 = 
	[\minmaxU] ^2.
	\end{equation*}
\end{lemma}

\begin{proof}
	Let
	\begin{equation*}
	U_n(x) = \begin{cases}
	\frac{1}{\sqrt{1-x^2}}\sin((n+1)\arccos(x)) & |x| \le 1 \\
	 \frac{1}{2\sqrt{x^2-1}}\left( (x + \sqrt{x^2-1})^{n+1} - (x - 
	 \sqrt{x^2-1})^{n+1}\right)
	  & |x| \ge 1
	\end{cases}
	\end{equation*}
	denote the order $n$ Chebyshev polynomial of the second kind. We 
	claim that $p^\star\in\Polys$ that solves the minimax problem $\min_{p 
		\in 
		\Polys} \max_{x\in [\alpha,\beta]} (x-\alpha)^{1/2}| 1 - x p(x) |$ is 
		given by
	\begin{equation*}
	1-xp^\star(x) = \frac{\minmaxU}{w(\beta)}\cdot
	U_{2n} \left( \sqrt{
	\frac{\kappa-x/\alpha}
	{\kappa-1} }\right),
	\end{equation*}
	where $\minmaxU = w(\beta)\left[ U_{2n} \left( 
	\sqrt{\frac{\kappa}{\kappa-1}} 
	\right)\right]^{-1}$ guarantees that the RHS has value 1 at $x=0$ and 
	therefore $p^\star$ is well defined (note that $U_{2n}(\cdot)$ is an even 
	polynomial and therefore $U_{2n}(\sqrt\cdot)$ is a polynomial of degree 
	$n$). For $x\in[\alpha,\beta]$, we have by the definition of $p^*$ and 
	the expression for $U_{2n}$,
	\begin{equation*}
	w(x)(1-xp^\star(x)) = \minmaxU \cdot 
	\sin\left( (2n+1)\arccos\left( \sqrt{\frac{\kappa-x/\alpha}{\kappa-1} 
	}\right)\right).
	\end{equation*}
	Therefore, we have that $w(x)|1-xp^\star(x)| \le \minmaxU$ for every 
	$x\in[\alpha,\beta]$, and moreover we have that $w(x_k)(1-x_k 
	p^\star(x_k)) = (-1)^k \cdot \minmaxU$, for the points $x_0, \ldots x_n 
	\in 
	[\alpha, \beta]$ satisfying
	\begin{equation*}
	\sqrt{\frac{\kappa-x_k/\alpha}{\kappa-1} }
	= \cos \left(\frac{\pi}{2}\cdot \frac{2k+1}{2n+1}\right).
	\end{equation*}
	Hence, the alternating signs argument from the proof of 
	Lemma~\ref{lem:cheby-first} holds here as well and we have that 
	$p^\star$ is optimal and that $\min_{p \in \Polys} \max_{x\in 
	[\alpha,\beta]} 
	{w(x)}| 1 - x 	p(x) | = \minmaxU$.

	To see the upper and lower bounds on $\minmaxU$, note that 
	$\minmaxU
	= \sqrt{\alpha}/\sinh((n+\half) \log( 1 + \frac{2}{\sqrt{\kappa}-1}))$, 
	that for $y\ge 0$, $\sinh(y) =\frac{1}{\sqrt{2}}\sqrt{\cosh(2y)-1}$ gives 
	$\half \sqrt{e^{2y}-2} \le 
	\sinh(y) \le \half \sqrt{e^{2y}-1}$, and that  (as in 
	Lemma~\ref{lem:cheby-first}) 
	$\frac{2}{z} \le \log\left( 1 + \frac{2}{z-1}\right) \le 
	\frac{2}{z-1}$.

	The final part of the lemma follows exactly as in 
	Lemma~\ref{lem:cheby-first}. 
\end{proof}

%% file: ub-proofs.tex
\section{Proofs from Section~\ref{sec:upper}}\label{sec:ub-proofs}

\subsection{Proof of auxiliary lemmas}\label{sec:prelim-proofs}

\lemLinCheby*

\begin{proof}
	This is an immediate consequence of Lemma~\ref{lem:cheby-first}, as
	\begin{equation*}
	\min_{p\in\Polys[t]}\max_{\alpha I \preceq M \preceq \beta I}\opnorm{I 
	- Mp(M)}  = 
	\min_{p\in\Polys[t]}\max_{\lambda\in [\alpha,\beta]}\left|1-\lambda 
	\cdot
	p(\lambda)\right| 
	=  \minmaxT[t].
	\end{equation*}
\end{proof}

\lemAGD*

\begin{proof}
	Let $g:\R^d \to \R$ be convex with $L$-Lipschitz gradient and let 
	$Q\subseteq\R^d$ be a convex set containing the point $0$. Consider 
	Nesterov's 
	accelerated gradient method for minimization of $g$, which 
	comprises the following recursion \cite[Scheme 
	(2.2.17)]{Nesterov04},
	\begin{gather*}
	x_{k+1} = \min_{x\in Q}\left\{ x^T \grad g(y_k) + \frac{L}{2}\norm{x 
		-y_k}^2 \right\} = \Pi_Q\left( y_k - \frac{1}{L}\grad g(y_k) \right)
	\\
	\alpha_{k+1}^2/(1-\alpha_{k+1}) = \alpha_k^2 \Rightarrow 
	\alpha_{k+1} = -\frac{\alpha_k^2}{2} 
	+ \frac{\alpha_k^2}{2}\sqrt{ 1+\frac{4}{\alpha_k^2}}
	\\
	y_{k+1} = x_{k+1} + \alpha_{k+1}(\alpha_k^{-1}-1)
	(x_{k+1} - x_k),
	\end{gather*}
	where $\Pi_Q(\cdot)$ is the Euclidean projection to $Q$. Letting 
	$\alpha_0=1$ and $y_0  = x_0=0$, and letting $x^\star$ denote any  
	minimizer of $g$ in $Q$, the analysis of~\citet[Corollary 2(b)]{Tseng08} 
	gives\footnote{ translating to the notation of~\cite{Tseng08}, take 
	$\phi(x,v) = g(x)$ and $P(x)$ to be the indicator of $Q$, 
	so that $q^P(\cdot) = g(x^\star)$, note that $\theta_k$ ($\alpha_k$ in 
	our notation) satisfies $\theta_k \le  2/(2+k)$. We discuss alternative 
	references for this result after the proof.},
	\begin{equation}\label{eq:tseng-bound}
	g(x_t) - g(x^\star) \le \frac{4L\max_{z\in Q}\norm{z}^2}{(t+1)^2}.
	\end{equation}
	
	Taking $g=\f[M,v]$ and $Q=B_r=\{x\mid\norm{x}\le r\}$, we note that 
	$\f[M,v]$ is convex with 
	$L\defeq\lambda_{\max}(M)$-Lipschitz gradient, and that the projection 
	step guarantees that $\norm{x_t}\le r$ for every $t$.  Therefore, to 
	establish the 	lemma it remains only to argue that $x_t$ as defined 
	above is in 
	$\Krylov[t][M,v]$; we shall see this by simple induction, whose basis is 
	$y_0, x_0 \in \Krylov[0][M,v] = \{0\}$. Assume now that $y_k, x_k \in 
	\Krylov[k][M,v]$ for some $k\ge0$. This implies
	\begin{equation*}
	y_k - \frac{1}{L}\grad g(y_k) = y_k - \frac{1}{L}A y_k - \frac{1}{L} v \in 
	\Krylov[k+1][M,v].
	\end{equation*}
	Further, note that projection to the Euclidean ball $B_r$ is simply scaling:
	\begin{equation*}
	\Pi_{Q}(z) = \Pi_{B_r}(z) = \frac{r}{\max\{r, \norm{z}\}} \cdot z,
	\end{equation*}
	and therefore $x_{k+1}\in \Krylov[k+1][M,v]$. Finally, $y_{k+1}$ is 
	simply a linear combination of $x_{k+1}$ and $x_k$ and therefore is 
	also in $\Krylov[k+1][M,v]$, concluding the induction and the proof.
\end{proof}

A bound similar to~\eqref{eq:tseng-bound} appears in Nesterov's earlier 
analysis~\cite[Theorem~2.2.3]{Nesterov04}, but with an the additional 
factor proportional to $g(0)-g(x\opt)$ which is not immediately upper 
bounded by $\half L \max_{z\in Q}\norms{z}^2$ due to the constraint 
$z\in Q$. 
The bound~\eqref{eq:tseng-bound} also appears in later 
work of~\citet{AllenOr17}.

\lemEigen*

\newcommand{\err}[1][t]{\mathsf{err}_{#1}}

\begin{proof}
	Let $\lambda\parind{1} \le \lambda\parind{2} \le  \cdots \le  
	\lambda\parind{d}$ denote the eigenvalues of $M$ and let $u_1, u_2, 
	\ldots, u_d$ denote their corresponding (orthonormal) eigenvectors. By 
	our assumption $\lambda\parind{1} = 0$ and we have also 
	$\lambda\parind{d} = \opnorm{M}$. We let 
	\begin{equation*}
	v\parind{i} \defeq u_i^T v
	\end{equation*}
	denote the component of $v$ in the eigenbasis of $M$. Define 
	\begin{equation*}
	\err \defeq 
	\min_{p\in\Polys[t]} \frac{ (p(M)v)^T M p(M)v }{\norm{p(M)v}^2 }
	= \min_{p\in\Polys[t]} \frac{ \sum_{i=1}^d v\parind{i}^2 
		p^2(\lambda\parind{i}) \lambda\parind{i}}
	{\sum_{i=1}^d v\parind{i}^2 p^2(\lambda\parind{i})},
	\end{equation*}
	and let $q\in\Polys[t]$ attain the minimum above. Setting $z_t = 
	q(M)v/\norm{q(M)v}$, we see that 
	\begin{equation*}
	\err = z_t^T M z_t = \frac{ \sum_{i=1}^d v\parind{i}^2 
		q^2(\lambda\parind{i}) \lambda\parind{i}}
	{\sum_{i=1}^d v\parind{i}^2 q^2(\lambda\parind{i})},
	\end{equation*}
	and so our 
	proof comprises of bounding $\err$ from above. 
	
	We invoke Lemma~\ref{lem:cheby-second} with $n=t-1$, $\alpha=\err$ 
	and $\beta=\lambda\parind{d}=\opnorm{M}$; let $\tilde{q}(x) = 
	1-xp^\star(x)\in\Polys[t]$ be the polynomial for which the Lemma 
	guarantees
	\begin{equation*}
	\max_{x\in [\err, \lambda\parind{d}]} (x-\err)^{1/2}|\tilde{q}(x)| = 
	\minmaxU[t-1].
	\end{equation*}
	By the optimality of $q$, we have that
	\begin{equation*}
	\err \le \frac{ \sum_{i=1}^d v\parind{i}^2 
		\tilde{q}^2(\lambda\parind{i}) \lambda\parind{i}}
	{\sum_{i=1}^d v\parind{i}^2 \tilde{q}^2(\lambda\parind{i})}.
	\end{equation*}
	Rearranging and noting that $\tilde{q}(\lambda\parind{1}) = \tilde{q}(0) = 
	1$, we obtain
	\begin{equation*}
	\err \le \sum_{i=2}^d \frac{v\parind{i}^2}{v\parind{1}^2}
	(\lambda\parind{i} - \err)\tilde{q}^2(\lambda\parind{i})
	\le \frac{\norm{v}^2-v\parind{1}^2}{v\parind{1}^2} \max_{\lambda\in 
		[\err, 
		\lambda\parind{d}]} (\lambda - \err) \tilde{q}^2(\lambda)
	= \left(\frac{\norm{v}^2}{v\parind{1}^2} -1\right)[\minmaxU[t-1]]^2.
	\end{equation*} 
	Lemma~\ref{lem:cheby-second} provides the bound
	\begin{equation*}
	[\minmaxU[t-1]]^2 \le \frac{4\err}{e^{2(2t-1)\sqrt{\err/\opnorm{M}}}-2}.
	\end{equation*}
	Substituting the upper bound into $\err \le 
	\big(\frac{\norm{v}^2}{v\parind{1}^2}-1\big)[\minmaxU[t-1]]^2$ and 
	rearranging 
	gives the result.
\end{proof}

\subsection{Proof of 
	Corollary~\ref{cor:tr-rand}}\label{sec:tr-rand-proof}
\corRand*
\begin{proof}
	Let $\pstr \in \argmin_{x\in\Krylov[t][A,\pb], \norm{x}\le R} 
	\f[A,\pb](x)$ be a solution to the perturbed problem. Since $v$ is a unit 
	vector, for any feasible $x$ we have
	\begin{align}\label{eq:tr-rand-pointwise-bound}
	\f(x) - \f(\str) & = \f[A,\pb](x) - \f[A,\pb](\str) + \sigma\cdot v^T (\str 
	- x)
	\le \f[A,\pb](x) - \f[A,\pb](\str) + 2\sigma R  \nonumber \\ &
	\le  \f[A,\pb](x) - \f[A,\pb](\pstr) + 2\sigma R,
	\end{align}
	and so it suffices to argue about the perturbed optimality gap 
	$\f[A,\pb](\pitertr_t) - \f[A,\pb](\str)$. 
	
	Applying the bound~\eqref{eq:tr-sublin-time} on the perturbed problem 
	gives us
	\begin{equation}\label{eq:tr-rand-pert-sublin-bound}
	\f[A,\pb](\pitertr_t) - \f[A,\pb](\pstr) \le 
	\frac{(\lmax - 
		\lmin)R^2}{(t-\half)^2} 
	\left[ 
	4 + \frac{\I_{\{\lmin < 0\}}}{2}
	\log^2\left(2\frac{\norms{\pb}}{|u_{\min}^T \pb|}\right)
	\right],
	\end{equation}
	and a simple argument on the density of $u_{\min}^T \pb$ 
	(cf.~\cite[Lemma 4.6]{CarmonDu16}) shows that
	\begin{equation}\label{eq:tr-rand-b1-bound}
	|u_{\min}^T \pb| \ge \frac{\sigma \cdot \delta}{\sqrt{d}}
	~~\mbox{with probability at least $1-\delta$}.
	\end{equation}
	Combining the bounds~\eqref{eq:tr-rand-pointwise-bound}, 
	\eqref{eq:tr-rand-pert-sublin-bound} and~\eqref{eq:tr-rand-b1-bound} 
	gives the result~\eqref{eq:tr-pert-sublin-time}.
\end{proof}

%% file: lb-proofs.tex
\section{Proof of lower bounds}\label{sec:lb-proof}

In what follows, we break Theorem~\ref{thm:lb} into two parts, 
one for the linear convergence lower bound~\eqref{eq:cr-lb-lin} and one for 
the sublinear lower bounds~\eqref{eq:cr-lb-sub-ncvx} 
and~\eqref{eq:cr-lb-sub-cvx}. We restate each sub-theorem in a way that 
clearly shows our control over problem-dependent parameters when 
constructing the hard problem instances. In our proofs we will make use of 
the following expression for the optimality gap in the cubic-regularization 
problem,
\begin{equation}\label{eq:fx-s-expression}
\fcu\left(x\right)-\fcu\left(\scu\right)=
\frac{1}{2}(x-\scu)^T\Acu (x-\scu) 
+\frac{\rho}{6}\left(\| \scu\| -\| x\| \right)^{2}\left(\| \scu\| +2\| x\| 
\right),
\end{equation}
where $\Acu = A + \rs I$.

\subsection{Proof of linear convergence lower bound}

\begin{customthm}{\ref{thm:lb}, part I}
	Let $\lmin, \lmax, \ltr, \Delta \in \R$ such that $\lmin \le \lmax$, $\ltr 
	> \max\{-\min,0\}$ and $R, \Delta>0$. For every $t \ge 1$ and every 
	$d>t$ there exists $A\in\R^{d\times d}$, $b\in\R^d$ and $\rho > 0$ 
	such that 
	\begin{itemize}
		\item all eigenvalues of $A$ are in $[\lmin, \lmax]$, 
		\item the solution $\scu = \argmin_{x\in\R^d}\fcu(x)$
		satisfies 
		$\rs = \ltr$,
		\item $\fcu(0) - \fcu(\scu) = \Delta$, and
	\end{itemize}
	\begin{equation*}
	\fcu(s) - \fcu(\scu) > 	 
	\left(1+\frac{\rs}{3(\rs + \lmin)} \right)^{-1} 
	\left[\fcu(0) - \fcu(\scu) \right]
	\exp\left\{
	-\frac{4t}{\sqrt{\frac{\rs + \lmax}{\rs +\lmin}}-1}
	\right\}.
	\end{equation*}
	for every $s\in\Krylov$.
\end{customthm}

\begin{proof}
	From Lemma~\ref{lem:cheby-first} with $\alpha = \ltr + \lmin$, 
	$\beta = \ltr + \lmax$ and $n=t$, we have that there exist $\xi_0, 
	\ldots, 
	\xi_t \in [\alpha, \beta]$ and probability distribution $\pi_0, \ldots, 
	\pi_{t}$ such that
	\begin{equation*}
	\min_{p \in \Polys[t]} \sum_{k=0}^{t} \pi_k (1-\xi_k p(\xi_k))^2 \ge 
	e^{-4t/(\sqrt{\kappa}-1)},
	\end{equation*}
	where $\kappa = \beta/\alpha = (\lmax+\ltr)/(\lmin+\ltr)$. We let $\xi$ 
	and $\sqrt{\pi}$ denote vectors with entries $\xi_0, \ldots, \xi_{t}$ and 
	$\sqrt{\pi_0}, \ldots, \sqrt{\pi_t}$ respectively.
	
	To construct the problem instance $(A,b,\rho)$ we assume without loss 
	of generality $d=t+1$, as higher dimensional instances can be obtained 
	by zero-padding a $(t+1)$-dimensional construction. We set
	\begin{equation*}
	A = \diag( \xi - \ltr )
	,~
	b = \mu \Atr^{1/2}\sqrt{\pi}
	~\mbox{and}~
	\rho =\ltr/ \norm{\Atr^{-1}b},
	\end{equation*}
	where we will choose $\mu > 0$ to set the value of 
	$\fcu(0)-\fcu(\scu)$. 
	First, we note that for any value of $\mu$ our choice of $\rho$ 
	guarantees that  $\norm{\Atr^{-1}b} = \ltr/\rho$, making $\scu = 
	-\Atr^{-1}b$ the unique global minimizer of $\fcu$. We therefore have 
	by equation~\eqref{eq:fx-s-expression}
	\begin{equation*}
	\fcu(0)-\fcu(\scu) = \half \scu^T \Atr \scu + \frac{\rs}{6}\norm{\scu}^2
	= \frac{\mu^2}{2}\left( 1 + \frac{\ltr}{3} \sqrt{\pi}^T \Atr^{-1} 
	\sqrt{\pi} \right),
	\end{equation*}
	so for every $\Delta > 0$ there is $\mu$ for which $\fcu(0)-\fcu(\scu) = 
	\Delta$. Noting that $\sqrt{\pi}^T \Atr^{-1} \sqrt{\pi} \le (\ltr + 
	\lmin)^{-1}\norm{\sqrt{\pi}}^2 = (\ltr + \lmin)^{-1}$, we also have
	\begin{equation*}
	\frac{\mu^2}{2} \ge \Delta \left(1+\frac{\ltr}{3(\ltr + \lmin)} \right)^{-1} .
	\end{equation*}
	
	Now, every $s\in\Krylov$ is of the form $s=-p(\Atr)b$ for 
	$p\in\Polys[t]$, and using 
	equation~\eqref{eq:fx-s-expression} again we have
	\begin{align*}
	\fcu(s)-\fcu(\scu) &\ge \half \norm{\Atr^{1/2}(s-\scu)}^2
	\stackrel{(a)}{=} \half\norm{(I-\Atr p(\Atr))\Atr^{-1/2}b}^2 \\ &
	\stackrel{(b)}{=} \frac{\mu^2}{2}\sum_{k=0}^{n} \pi_k (1-\xi_k 
	p(\xi_k))^2 
	\stackrel{(c)}{\ge} 
	\frac{\mu^2}{2} e^{-4t/(\sqrt{\kappa}-1)},
	\end{align*}
	where in $(a)$ we substituted $s=-p(\Atr)b$ and $\scu=-\Atr^{-1}b$, in 
	$(b)$ we used our construction of $A$ and $b$, and in $(c)$ we used the 
	guarantee from Lemma~\ref{lem:cheby-first}. The result follows from 
	substituting our lower bound on $\mu^2$ and recalling that $\ltr = \rs$.
\end{proof}

\subsection{A lower bound for finding eigenvectors}
The ``non-convex'' lower bound is in its heart a statement about the 
difficulty of approximating an extremal eigenvector in a Krylov subspace, 
which we state explicitly here. The proof of the lemma consists of applying 
``in reverse'' the same polynomial approximation result 
(Lemma~\ref{lem:cheby-second}) that~\citet{KuczynskiWo92} use for 
proving upper bounds on finding eigenvector with the Lanczos method 
(which we state as Lemma~\ref{lem:eigenvec-tight}).

\begin{restatable}[Finding 
	eigenvectors: lower bound]{lemma}{lemEigenLB}\label{lem:eigenvec-lb}
	For every $d>0$, vector $v\in \R^d$, unit vector $u\in\R^d$ and $t<d$, 
	there exists matrix 
	$M\in\R^{d\times d}$ such that $M \succeq 0$, $M u = 0$, and for 
	every $z \in \Krylov[t][M,v]$, 
	\begin{equation*}
	\frac{z_t^T M z_t}{\opnorm{M}\norm{z_t}^2} \ge 
	\min\left\{\frac{1}{4}, 
	\frac{1}{64(t-\half)^2}
	\log^2\left(-3+4\frac{\norm{v}^2}{(u^T v)^2}\right)\right\}.
	\end{equation*}
\end{restatable}

\begin{proof}
	We take $\opnorm{M}=1$ without loss of generality; results for arbitrary 
	norms of $M$ follow by scaling the construction below. Define
	\begin{equation}\label{eq:eigenvec-lb-err-def}
	\err \defeq 
	\min\left\{\frac{1}{4}, 
	\frac{1}{64(t-\half)^2}
	\log^2\left(-3+4\frac{\norm{v}^2}{(u^T v)^2}\right)\right\}.
	\end{equation}
	We apply Lemma~\ref{lem:cheby-second} with $n=t-1$, $\alpha = \err$ 
	and $\beta=1$, to obtain $\xi_1, \ldots, \xi_t \in [\err, 1]$ and 
	probability 
	distribution $\pi_1, \ldots, \pi_t$ such that
	\begin{equation}\label{eq:cheby-second-lb-applied}
	\min_{p \in \Polys[t-1]} \sum_{k=1}^{t} \pi_k (\xi_k - \err)(1-\xi_k 
	p(\xi_k))^2 
	\ge  \frac{4\err}{e^{2(2t-1)/(\frac{1}{\sqrt{\err}}-1)}-1}.
	\end{equation}
	
	We assume without loss of generality that $d=t+1$ (otherwise we 
	zero-pad), and construct $M$ as follows. First, we take the eigenvalues 
	of 
	$M$ 
	to be $0, \xi_1, \ldots, \xi_t$, satisfying $0\preceq M \preceq I$. Next, 
	we 
	let $u$ be the eigenvector of $M$ 
	corresponding to eigenvalue $0$, satisfying $Mu=0$. Finally, for 
	$i=1,\ldots, t$ we choose the 
	eigenvector $u_i$ corresponding to eigenvalue $\xi_i$ such that $(u_i^T 
	v)^2 = \pi_i (\norm{v}^2 - (u^T v)^2)$. 
	
	Assume by contradiction
	\begin{equation}\label{eq:eigenvec-lb-contra-assumption}
	\min_{z\in\Krylov[t][M,v]} \frac{ z^T M z
	}{\norm{z}^2 } < \err,
	\end{equation}
	and let $q\in \Polys[t]$ be be such that
	\begin{equation*}
	\frac{ \sum_{i=1}^t \xi_i q^2(\xi_i) (u_i^T v)^2}
	{ q^2(0)(u^T v)^2 + \sum_{i=1}^t q^2(\xi_i) (u_i^T v)^2}
	= \frac{ (q(M)v)^T M q(M)v 
	}{\norm{q(M)v}^2 }
	= \min_{z\in\Krylov[t][M,v]} \frac{ z^T M z
	}{\norm{z}^2 } < \err.
	\end{equation*}
	Rearranging, using $(u_i^T 
	v)^2 = \pi_i (\norm{v}^2 - (u^T v)^2)$, and letting $\tilde{q}(x) = 
	q(x)/q(0)$, 
	we have that
	\begin{equation*}
	\err  > \left( \frac{\norm{v}^2}{(u^T v)^2} - 1\right)
	\sum_{i=1}^{t} \pi_i (\xi_i - \err) \tilde{q}^2(\xi_i)
	\ge 
	\left( \frac{\norm{v}^2}{(u^T v)^2} - 1\right) \frac{4
		\err}{e^{2(2t-1)/(\frac{1}{\sqrt{\err}}-1)}-1}.
	\end{equation*}
	where in the last transition we used that $\tilde{q}(0)=1$ and therefore it 
	is 
	of the form $1-xp(x)$ for some $p\in\Polys[t-1]$, so the lower 
	bound~\eqref{eq:cheby-second-lb-applied} applies. Rearranging gives
	\begin{equation*}
	\err > h\left(\frac{1}{16(t-\half)^2}
	\log^2\left(-3+4\frac{\norm{v}^2}{(u^T v)^2}\right)\right)
	,~~
	h(x) = \frac{x}{(1+\sqrt{x})^2}.
	\end{equation*}
	Using $h(x) \ge \frac{1}{4}\min\{1, x\}$ and the 
	definition~\eqref{eq:eigenvec-lb-err-def} of $\err$, we see that the 
	above 
	bound 
	gives the contradiction  $\err > \err$ and therefore 
	assumption~\eqref{eq:eigenvec-lb-contra-assumption} must be false 
	and 
	we have the desired result $\min_{z\in\Krylov[t][M,v]} \frac{ z^T M z
	}{\norm{z}^2 } \ge \err$.
\end{proof}

\subsection{Proof of sublinear convergence lower bound}

\begin{customthm}{\ref{thm:lb}, part II}
	Let $\lmin, \lmax, R, \tau \in \R$ such that $\lmin\le\lmax$, $\tau \ge 
	1$ 
	and $R>0$. For 
	every $t \ge 1$ and every 
	$d>t$ there exists $A\in\R^{d\times d}$, $b\in\R^d$ and $\rho > 0$ 
	such that 
	\begin{itemize}
		\item all eigenvalues of $A$ are in $[\lmin, \lmax]$, 
		\item the solution $\scu = \argmin_{x\in\R^d}\fcu(x)$
		satisfies $\norm{\scu} = R$, 
		\item there exists unit eigenvector $u_{\min}$ such that $u_{\min}^T 
		A u_{\min} = \lmin$ and $\frac{\norm{b}}{|u_{\min}^T b|} = \tau$, and
	\end{itemize}
	\begin{equation*}
	\fcu(s) - \fcu(\scu) > 
	\min\left\{
	\lmax^{-}-\lmin~,~
	\frac{\lmax-\lmin}{16(t-\half)^2}
	\log^2\left(\frac{\norm{b}^2}{( u_{\min}^T b)^2}\right)
	\right\}\frac{\norm{\scu}^2}{32},
	\end{equation*}
	where $\lmax^{-}=\min\{\lmax,0\}$, and
	\begin{equation*}
	\fcu(s) - \fcu(\scu) > 
	\frac{(\lmax - \lmin)\norm{\scu}^2}{16(t+\half)^2}.
	\end{equation*}
	for every $s\in\Krylov$.
\end{customthm}

\renewcommand{\err}[1][t]{\epsilon_{#1}}

\begin{proof}
	We begin with the first, ``non-convex'' bound, which is essentially a 
	reduction to the eigenvector problem. Here we assume $\lmin \le 0$ as 
	otherwise the lower bound is vacuous. We use 
	Lemma~\ref{lem:eigenvec-lb} to construct $M\in\R^{d\times d}$ 
	and unit vectors $u_{\min}, v \in R^d$ such that $M\succeq 0$, 
	$\opnorm{M}=\lmax - \lmin$, $M u_{\min} = 0$, $\norm{v}/|u_{\min}^T 
	v| 
	= \tau$ and for every $z\in \Krylov[t][M,v]$
	\begin{align}\label{eq:cr-sub-lb-err-def}
	\frac{z^T M z}{\norm{z}^2} & \ge \frac{\lmax - \lmin}{4}
	\min\left\{1 \,,\, 
	\frac{1}{16(t-\half)^2}
	\log^2\left(-3+4\frac{\norm{v}^2}{( u_{\min}^T v)^2}\right)\right\} 
	\nonumber \\ &
	\ge 
	\frac{1}{4}\min\left\{\lmax^{-} -\lmin \,,\, 
	\frac{\lmax - \lmin}{16(t-\half)^2}
	\log^2\left(\frac{\norm{v}^2}{( u_{\min}^T v)^2}\right)\right\} 
	\defeq
	\err,
	\end{align}
	where $\lmax^{-}=\min\{\lmax,0\}$.
	We let $\eps > 0$ be a parameter to be specified later. We let
	\begin{equation*}
	\ltr = -\lmin + \eps
	\end{equation*}
	and construct the cubic regularization instance as follows
	\begin{equation*}
	A = M + \lmin I
	~,~
	b = \frac{R}{\norm{\Atr^{-1}v}}v
	~,~
	\rho = \ltr / R.
	\end{equation*}
	The solution for this instance is unique and satisfies $\scu = 
	-\Atr^{-1}b = 
	- R \Atr^{-1}v / \norm{\Atr^{-1}v}$ so that $\norm{\scu} = R$, and 
	moreover we note that $\norm{b} \to 0$ as $\eps\to 0$. For every 
	$s\in\Krylov[t][M,v]=\Krylov[t][A,b]$,
	\begin{equation*}
	\fcu(s) = \half s^T A s + b^T s + \frac{\rho}{3}\norm{s}^3
	\ge -\norm{b}\norm{s} + \half(\lmin + \err)\norm{s}^2 + 
	\frac{\rho}{3}\norm{s}^3.
	\end{equation*}
	The RHS above is minimal for
	\begin{equation*}
	\norm{s} = \tilde{R} \defeq -\frac{\lmin + \err}{2\rho} +  
	\sqrt{\left( \frac{\lmin + \err}{2\rho} \right)^2 + \frac{\norm{b}}{\rho}}
	\le \frac{-\lmin - \err}{\rho}+ 
	\sqrt{\frac{\norm{b}}{\rho}},
	\end{equation*}
	where the bound holds since our definition of $\err$ implies $\err \le 
	-\lmin$ and so $-\lmin-\err \ge 0$.
	The minimum value of the RHS satisfies
	\begin{equation}\label{eq:cr-sub-lb-kryl-lb}
	\fcu(s) \ge -\frac{2}{3}\norm{b}\tilde{R} - \frac{1}{6}(-\lmin - \err) 
	\tilde{R}^2.
	\end{equation}
	Taking without loss of generality $u_{\min}^T b \le 0$ and using $\rho 
	= 
	\ltr / R$, and $\ltr = -\lmin + \eps$, we have
	\begin{equation}\label{eq:cr-sub-lb-opt-lb}
	\fcu(\scu) \le \fcu(R \cdot u_{\min}) \le \half \lmin R^2 + \frac{1}{3}\ltr 
	R^2 
	= \frac{1}{6} \lmin R^2 + \frac{\eps}{3} R^2.
	\end{equation}
	Recall that $\norm{b}\to 
	0$ 
	as $\eps\to0$, and take $\eps>0$ 
	sufficiently small so that
	\begin{equation*}
	\eps < \err/24
	~~
	\mbox{and}
	~~
	\norm{b} \le \min\{\err R/24 , \err^2/\rho \},
	\end{equation*}
	which implies also
	\begin{equation*}
	\tilde{R} \le \frac{-\lmin - \err}{\rho}
	+ \frac{\err}{\rho} = \frac{-\lmin}{\rho} \le \frac{\ltr}{\rho} = R.
	\end{equation*}
	Using $\tilde{R}\le R$, we may replace $\tilde R$ with $R$ in the 
	bound~\eqref{eq:cr-sub-lb-kryl-lb}, and combining this  
	with~\eqref{eq:cr-sub-lb-opt-lb} and the bounds on $\norm{b}$ and 
	$\eps$ we obtain
	\begin{equation*}
	\fcu(s)-\fcu(\scu) \ge 
	\frac{\err}{6}R^2 - \frac{2\norm{b}}{3}R - 
	\frac{\eps}{3}R^2 \ge \frac{\err}{8} R^2.
	\end{equation*}
	Recalling $\norm{\scu} = R$ and the 
	definition~\eqref{eq:cr-sub-lb-err-def} 
	of $\err$, we get the desired ``non-convex'' lower bound.

	To derive the alternative, ``convex'' lower bound, 
	we again let 
	$0<\eps<\lmax-\lmin$ be a parameter to be determined, and we apply 
	Lemma~\ref{lem:cheby-second} with $n=t$, $\alpha=\eps$, 
	$\beta=\lmax-\lmin$ to obtain points $\xi_0, \ldots, \xi_t \in [0, 
	\lmax-\lmin]$ and probability masses $\pi_0, \ldots, \pi_t$ such that
	\begin{equation*}
	\min_{p \in \Polys} \sum_{k=0}^{n} \pi_k (\xi_k - \eps)(1-\xi_k 
	p(\xi_k))^2 = \left[\minmaxU[t][\frac{\lmax-\lmin}{\eps}] \right]^2.
	\end{equation*}
	To construct the hard instance we again set
	\begin{equation*}
	\ltr = -\lmin + \eps.
	\end{equation*}
	Letting $\xi$ and $\sqrt{\pi}$ denote vectors with entries $\xi_i$ and 
	$\sqrt{\pi}_i$, we set
	\begin{equation*}
	A = \diag(\xi - \ltr)
	~,~
	b = R\cdot\Atr\sqrt{\pi}
	~,~
	\rho = \ltr / R.
	\end{equation*}
	Again we have that $\scu = -\Atr^{-1}b$ is the unique solution and 
	$\norm{\scu} = R\norm{\sqrt{\pi}}=R$. Let $s\in\Krylov$, then 
	\begin{equation*}
	s = -p(\Atr)b = p(\Atr)\Atr \scu = -R p(\Atr)\Atr \sqrt{\pi}
	\end{equation*}
	for some $p\in\Polys[t]$. By equality~\eqref{eq:fx-s-expression} we 
	have
	\begin{align*}
	\fcu(s) - \fcu(\scu) & \ge \frac{1}{2}\norm{\Atr^{1/2}(s-\scu)}^2
	=
	\frac{R^2}{2} 
	\sum_{k=0}^t \pi_k \xi_k (1-\xi_k p(\xi_k) )^2 \\ &
	\ge \frac{R^2}{2} \sum_{k=0}^t \pi_k (\xi_k -\eps) (1-\xi_k p(\xi_k) )^2
	= \frac{R^2}{2}  \left[\minmaxU[t][\frac{\lmax-\lmin}{\eps}] \right]^2.
	\end{align*}
	Note that
	\begin{equation*}
	\lim_{\eps\to 0} \minmaxU[t][\frac{\lmax-\lmin}{\eps}]
	=
	\frac{\sqrt{\lmax-\lmin}}{2t+1}.
	\end{equation*}
	Therefore, we can choose $\eps$ sufficiently small so that
	\begin{equation*}
	\left[\minmaxU[t][\frac{\lmax-\lmin}{\eps}] \right]^2 
	\ge \frac{\lmax-\lmin}{2(2t+1)^2},
	\end{equation*}
	which gives the proof for the ``convex'' lower bound, as 
	$\norm{\scu}=R$.
\end{proof}

%% file: experiment-details.tex
\section{Numerical experiment details}\label{sec:exp-details}

\paragraph{Random problem generation, $\kappa < \infty$}
We generate random cubic 
regularization instances $(A, b, \rho)$ as 
follows. We take $\lmax = 1$ and draw $\lmin\sim U[-1,-0.1]$, where 
$U[a,b]$ denotes the uniform distribution on $[a,b]$. We then fix two 
eigenvalues of $A$ to be $\lmin,\lmax$ and draw the other $d-2$ 
eigenvalues independently from $U[\lmin, \lmax]$. We then take $A$ to be 
diagonal with said eigenvalues. This is without much loss of generality (as 
the Krylov subspace method is rotationally invariant), and it allows us to 
quickly compute matrix-vector products, whose computation nevertheless 
accounts for much of the experiment running time when using 
$d=10^{6}$. 

For a desired condition number $\kappa$, we let
\begin{equation*}
\ltr \defeq \frac{\lmax - \kappa\lmin}{\kappa - 1}
\end{equation*} 
and as usual denote $\Atr = A + \ltr I$. 
To generate $b$, $\rho$, we draw a standard normal $d$-dimensional 
vector 
$v\sim \mc{N}(0; I)$ and let
\begin{equation*}
b =  \sqrt{\frac{2}{v^T \Atr^{-1} v + 
		\frac{\ltr}{3} v^T\Atr^{-2} v} }\cdot  v
~,~
\rho = \frac{\ltr}{\norms{\Atr^{-1}b}},
\end{equation*}
The above choice of $b$ and $\rho$ guarantees that $\rho 
\norm{\Atr^{-1}b} = \ltr$ and therefore $\scu = -\Atr^{-1}b$ is the unique 
solution and the problem condition number satisfies
\begin{equation*}
\frac{\lmax+\rs}{\lmin+\rs} = \frac{\lmax + \ltr}{\lmin + \ltr} = \kappa
\end{equation*}
as desired. Moreover, our scaling of $b$ guarantees that
\begin{equation*}
\fcu(0)-\fcu(\scu) = \half (\scu)^T \Atr \scu + \frac{\rho}{6}\norm{\scu}^3
=
\half \left(b^T \Atr^{-1} b + \frac{\ltr}{3}b^T \Atr^{-2} b\right) =1.
\end{equation*}

Our technique for generating $(A,b,\rho)$ is similar to the one we used 
in~\cite{CarmonDu16} to test gradient descent for cubic regularization. 
The main difference is that in~\cite{CarmonDu16} the value of $\rho$ is 
fixed 
and consequently there is no control over the initial optimality gap. 

For every value of $\kappa$, we generate 5,000 problem instances 
independently as described above.

\paragraph{Random problem generation, $\kappa = \infty$} We let 
$A=\diag(\lambda)$ where $\lambda_1 = \lmin = -0.5$, $\lambda_d = 
\lmax = 0.5$ and $\lambda_2, \ldots, \lambda_{d-1}$ are drawn i.i.d. from 
$U[\lmin + \gamma, \lmax]$ where we take the eigen-gap 
$\gamma=10^{-4}$ and $d=10^6$. 

As $\kappa=\infty$, we let
\begin{equation*}
\ltr = -\lmin
\end{equation*}
and denote $\hat{A}_{\ltr} \defeq \diag(\lambda_2+\ltr, \ldots, 
\lmax+\ltr)$. 
We generate $b$ and $\rho$ by drawing a standard normal 
$(d-1)$-dimensional vector $v$, and letting
\begin{equation*}
b_1 = 0
~,~
b_{2:d} = \sqrt{\frac{2}{v^T \hat{A}_{\ltr}^{-1} v + 
		(1+\tau^2)\frac{\ltr}{3} v^T\hat{A}_{\ltr}^{-2} v} } v
~,~
\rho = \frac{\ltr}{\norms{\hat{A}_{\ltr}^{-1}b_{2:d}}\sqrt{1+\tau^2}},
\end{equation*}
where $\tau$ is a parameter that determines the weight of the eigenvector 
corresponding to $\lmin$ in the solution (when $\tau = \infty$ we have a 
pure eigenvector instance); we take $\tau=10$. A global minimizer $\scu$ 
of the problem instance $(A,b,\rho)$ generated above has the form,
\begin{equation*}
[\scu]_1 = \pm \tau \norms{\hat{A}_{\ltr}^{-1}b_{2:d}}
~,~
[\scu]_{2:d} = -\hat{A}_{\ltr}^{-1}b_{2:d}.
\end{equation*}
As in the case $\kappa < \infty$, it is easy to verify that the scaling of $b$ 
guarantees 
$\fcu(0)-\fcu(\scu) = 1$.

When $\kappa=\infty$, the choice of eigen-gap $\gamma$ strongly affects 
optimization performance. We explore this in Figure~\ref{fig:exp-gap}, 
which 
repeats the 
experiment described above with different values of $\gamma$ (and 
$d=10^5$). As 
seen in the figure, the non-randomized Krylov subspace solution becomes 
more suboptimal as $\gamma$ increases. Moreover, randomization 
``kicks-in'' after roughly $\log d / \sqrt{\gamma}$ iterations, when  
eigen-gap-dependent linear convergence begins.

To create each plot, we draw 10 independent problem instances from the 
distribution described above, and for each problem instance run each 
randomization approach with 50 different random seeds; we observe that 
sampling problem instances and sampling randomization seeds contribute 
similar amount of variation to the final ensemble of results.

\begin{figure}
	\centering
	\includegraphics[width=\nips{\columnwidth}\arxiv{0.975\columnwidth}]
	{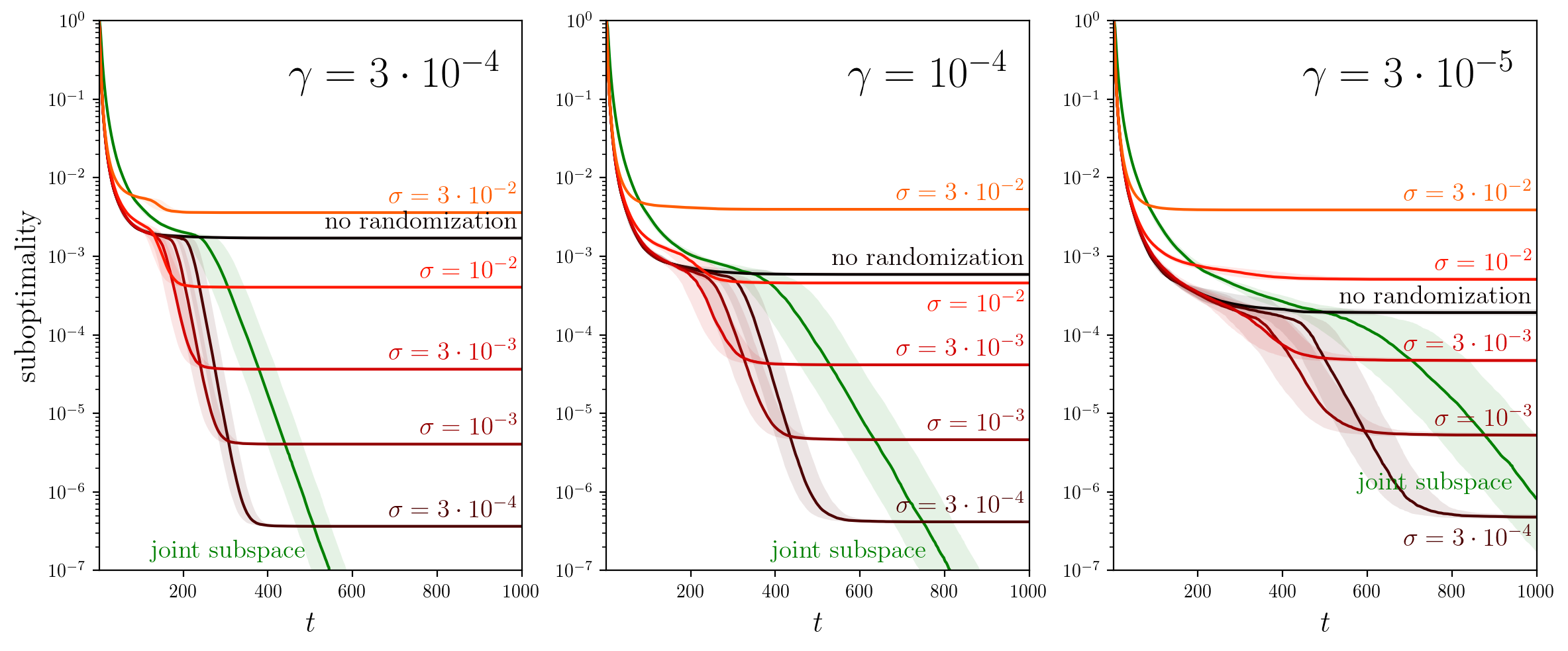}
	\vspace{-0.5cm}
	\caption{\label{fig:exp-gap}Optimality gap of 
		Krylov subspace solutions on random 
		cubic-regularization problems, versus subspace 
		dimension $t$. Each plot shows result for problem instances with a 
		different eigen-gap $\gamma = (\lmax - \lmin)/(\lambda_2 - 
		\lmin)$, where $\lambda_2$ is the smallest eigenvalue larger than 
		$\lmin$. Each line represents median suboptimality, and shaded 
		regions represent inter-quartile range. Different lines 
		correspond to different randomization settings.
	}
\end{figure}

\paragraph{Hardness of generated problems}
It is well known that the performance of subspace methods improves 
dramatically when the eigenvalues of $A$ are 
clustered~\cite{TrefethenBa97}. Taking the 
eigenvalues of $A$ to be uniformly distributed produces very little 
clustering, making the instances we draw somewhat hard. However, 
examining the proof of the lower bound~\eqref{eq:cr-lb-lin} we see that 
the worst case eigenvalues are of the form $\lambda_k = \lmin + 
(\lmax-\lmin)\sin^2 \theta_k$ where $\theta_1, \ldots \theta_d$ are 
equally spaced in $[0, \pi/2]$. This is fairly different from a uniform 
distribution (asymptotically as $d\to\infty$ it becomes an arcsine 
distribution), and consequently we think that uniformly distributing the 
eigenvalues makes for a challenging  but not quite adversarial test case.

\paragraph{Computing Krylov subspace solutions}
We use the Lanczos process to obtain a tridiagonal representation of $A$ as 
described in Section~\ref{sec:lanczos}. To obtain full optimization traces 
we solve equation~\eqref{eq:cr-lambda-search} after every Lanczos 
iteration, warm-starting $\lambda$ with the solution from the previous 
step and the minimum eigenvalue of the current tridiagonal matrix. 
We use the Newton method described by~\citet[Algorithm 
6.1]{CartisGoTo11} to solve the equation~\eqref{eq:cr-lambda-search} in 
the Krylov subspace. For the $\kappa < \infty$ experiment, we stop 
the process when
$|\norm{A_{\lambda}^{-1}b}-\lambda/\rho| < 10^{-12}$ or after 25 
tridiagonal system solves are computed. For the $\kappa = \infty$ 
experiment we allow up to 100 system solves.

%% file: krylov_paper.bbl
\begin{thebibliography}{39}
\providecommand{\natexlab}[1]{#1}
\providecommand{\url}[1]{\texttt{#1}}
\expandafter\ifx\csname urlstyle\endcsname\relax
  \providecommand{\doi}[1]{doi: #1}\else
  \providecommand{\doi}{doi: \begingroup \urlstyle{rm}\Url}\fi

\bibitem[Agarwal et~al.(2017)Agarwal, Allen-Zhu, Bullins, Hazan, and
  Ma]{AgarwalAlBuHaMa17}
N.~Agarwal, Z.~Allen-Zhu, B.~Bullins, E.~Hazan, and T.~Ma.
\newblock Finding approximate local minima faster than gradient descent.
\newblock In \emph{Proceedings of the Forty-Ninth Annual ACM Symposium on the
  Theory of Computing}, 2017.

\bibitem[{Allen-Zhu} and Orecchia(2017)]{AllenOr17}
Z.~{Allen-Zhu} and L.~Orecchia.
\newblock Linear coupling: An ultimate unification of gradient and mirror
  descent.
\newblock In \emph{Proceedings of the 8th Innovations in Theoretical Computer
  Science}, ITCS~'17, 2017.

\bibitem[Blanchet et~al.(2016)Blanchet, Cartis, Menickelly, and
  Scheinberg]{BlanchetCaMeSc16}
J.~Blanchet, C.~Cartis, M.~Menickelly, and K.~Scheinberg.
\newblock Convergence rate analysis of a stochastic trust region method for
  nonconvex optimization.
\newblock \emph{arXiv:1609.07428 [math.OC]}, 2016.

\bibitem[Cameron~Musco(2017)]{MuscoMuSi17}
A.~S. Cameron~Musco, Christopher~Musco.
\newblock Stability of the {L}anczos method for matrix function approximation.
\newblock \emph{arXiv:1708.07788 [cs.DS]}, 2017.

\bibitem[Carmon and Duchi(2016)]{CarmonDu16}
Y.~Carmon and J.~C. Duchi.
\newblock Gradient descent efficiently finds the cubic-regularized non-convex
  {N}ewton step.
\newblock \emph{arXiv:1612.00547 [math.OC]}, 2016.

\bibitem[Carmon et~al.(2017{\natexlab{a}})Carmon, Duchi, Hinder, and
  Sidford]{CarmonDuHiSi17}
Y.~Carmon, J.~C. Duchi, O.~Hinder, and A.~Sidford.
\newblock Convex until proven guilty: dimension-free acceleration of gradient
  descent on non-convex functions.
\newblock In \emph{Proceedings of the 34th International Conference on Machine
  Learning}, 2017{\natexlab{a}}.

\bibitem[Carmon et~al.(2017{\natexlab{b}})Carmon, Duchi, Hinder, and
  Sidford]{CarmonDuHiSi17lii}
Y.~Carmon, J.~C. Duchi, O.~Hinder, and A.~Sidford.
\newblock Lower bounds for finding stationary points {II}: First order methods.
\newblock \emph{arXiv:1711.00841 [math.OC]}, 2017{\natexlab{b}}.

\bibitem[Carmon et~al.(2018)Carmon, Duchi, Hinder, and Sidford]{CarmonDuHiSi18}
Y.~Carmon, J.~C. Duchi, O.~Hinder, and A.~Sidford.
\newblock Accelerated methods for non-convex optimization.
\newblock \emph{SIAM Journal on Optimization}, 28\penalty0 (2):\penalty0
  1751--1772, 2018.
\newblock URL \url{https://arXiv.org/abs/1611.00756}.

\bibitem[Cartis et~al.(2011)Cartis, Gould, and Toint]{CartisGoTo11}
C.~Cartis, N.~I.~M. Gould, and P.~L. Toint.
\newblock Adaptive cubic regularisation methods for unconstrained optimization.
  {P}art {I}: motivation, convergence and numerical results.
\newblock \emph{Mathematical Programming, Series A}, 127:\penalty0 245--295,
  2011.

\bibitem[Coakley and Rokhlin(2013)]{CoakleyRo13}
E.~S. Coakley and V.~Rokhlin.
\newblock A fast divide-and-conquer algorithm for computing the spectra of real
  symmetric tridiagonal matrices.
\newblock \emph{Applied and Computational Harmonic Analysis}, 34\penalty0
  (3):\penalty0 379--414, 2013.

\bibitem[Conn et~al.(2000)Conn, Gould, and Toint]{ConnGoTo00}
A.~R. Conn, N.~I.~M. Gould, and P.~L. Toint.
\newblock \emph{Trust Region Methods}.
\newblock MPS-SIAM Series on Optimization. SIAM, 2000.

\bibitem[Cullum and Donath(1974)]{CullumDo74}
J.~Cullum and W.~E. Donath.
\newblock A block {L}anczos algorithm for computing the q algebraically largest
  eigenvalues and a corresponding eigenspace of large, sparse, real symmetric
  matrices.
\newblock In \emph{Decision and Control including the 13th Symposium on
  Adaptive Processes, 1974 IEEE Conference on}, volume~13, pages 505--509.
  IEEE, 1974.

\bibitem[Druskin and Knizhnerman(1991)]{DruskinKn91}
V.~Druskin and L.~Knizhnerman.
\newblock Error bounds in the simple {L}anczos procedure for computing
  functions of symmetric matrices and eigenvalues.
\newblock \emph{U.S.S.R.\ Computational Mathematics and Mathematical Physics},
  31\penalty0 (7):\penalty0 970--983, 1991.

\bibitem[Golub and Loan(1989)]{GolubVa89}
G.~Golub and C.~V. Loan.
\newblock \emph{Matrix computations}.
\newblock John Hopkins University Press, 1989.

\bibitem[Golub and Underwood(1977)]{Golub77}
G.~H. Golub and R.~Underwood.
\newblock The block {L}anczos method for computing eigenvalues.
\newblock In \emph{Mathematical software}, pages 361--377. Elsevier, 1977.

\bibitem[Gould et~al.(2003)Gould, Orban, and Toint]{GouldOrTo03}
N.~I. Gould, D.~Orban, and P.~L. Toint.
\newblock {GALAHAD}, a library of thread-safe {F}ortran 90 packages for
  large-scale nonlinear optimization.
\newblock \emph{{ACM} Transactions on Mathematical Software (TOMS)},
  29\penalty0 (4):\penalty0 353--372, 2003.

\bibitem[Gould et~al.(1999)Gould, Lucidi, Roma, and Toint]{GouldLuRoTo99}
N.~I.~M. Gould, S.~Lucidi, M.~Roma, and P.~L. Toint.
\newblock Solving the trust-region subproblem using the {L}anczos method.
\newblock \emph{SIAM Journal on Optimization}, 9\penalty0 (2):\penalty0
  504--525, 1999.

\bibitem[Griewank(1981)]{Griewank81}
A.~Griewank.
\newblock The modification of {N}ewton's method for unconstrained optimization
  by bounding cubic terms.
\newblock Technical report, Technical report NA/12, 1981.

\bibitem[Hazan and Koren(2016)]{HazanKo16}
E.~Hazan and T.~Koren.
\newblock A linear-time algorithm for trust region problems.
\newblock \emph{Mathematical Programming, Series A}, 158\penalty0 (1):\penalty0
  363--381, 2016.

\bibitem[Hestenes and Stiefel(1952)]{HestenesSt52}
M.~Hestenes and E.~Stiefel.
\newblock Methods of conjugate gradients for solving linear systems.
\newblock \emph{Journal of Research of the National Bureau of Standards},
  49\penalty0 (6), 1952.

\bibitem[Ho-Nguyen and K{\i}l{\i}n\k{c}-Karzan(2016)]{Ho-NguyenKi16}
N.~Ho-Nguyen and F.~K{\i}l{\i}n\k{c}-Karzan.
\newblock A second-order cone based approach for solving the trust-region
  subproblem and its variants.
\newblock \emph{arXiv:1603.03366 [math.OC]}, 2016.

\bibitem[Jin et~al.(2017)Jin, Netrapalli, and Jordan]{JinNeJo17}
C.~Jin, P.~Netrapalli, and M.~I. Jordan.
\newblock Accelerated gradient descent escapes saddle points faster than
  gradient descent.
\newblock \emph{arXiv:1711.10456 [cs.LG]}, 2017.

\bibitem[Kohler and Lucchi(2017)]{KohlerLu17}
J.~M. Kohler and A.~Lucchi.
\newblock Sub-sampled cubic regularization for non-convex optimization.
\newblock In \emph{Proceedings of the 34th International Conference on Machine
  Learning}, 2017.

\bibitem[Kuczynski and Wozniakowski(1992)]{KuczynskiWo92}
J.~Kuczynski and H.~Wozniakowski.
\newblock Estimating the largest eigenvalue by the power and {L}anczos
  algorithms with a random start.
\newblock \emph{SIAM Journal on Matrix Analysis and Applications}, 13\penalty0
  (4):\penalty0 1094--1122, 1992.

\bibitem[Lenders et~al.(2018)Lenders, Kirches, and Potschka]{LendersKiPo18}
F.~Lenders, C.~Kirches, and A.~Potschka.
\newblock trlib: A vector-free implementation of the {GLTR} method for
  iterative solution of the trust region problem.
\newblock \emph{Optimization Methods and Software}, 33\penalty0 (3):\penalty0
  420--449, 2018.

\bibitem[Nemirovski(1994)]{Nemirovski94}
A.~Nemirovski.
\newblock Efficient methods in convex programming.
\newblock Technion: The Israel Institute of Technology, 1994.

\bibitem[Nemirovski and Yudin(1983)]{NemirovskiYu83}
A.~Nemirovski and D.~Yudin.
\newblock \emph{Problem Complexity and Method Efficiency in Optimization}.
\newblock Wiley, 1983.

\bibitem[Nesterov(2004)]{Nesterov04}
Y.~Nesterov.
\newblock \emph{Introductory Lectures on Convex Optimization}.
\newblock Kluwer Academic Publishers, 2004.

\bibitem[Nesterov and Polyak(2006)]{NesterovPo06}
Y.~Nesterov and B.~Polyak.
\newblock Cubic regularization of {N}ewton method and its global performance.
\newblock \emph{Mathematical Programming, Series A}, 108:\penalty0 177--205,
  2006.

\bibitem[Nocedal and Wright(2006)]{NocedalWr06}
J.~Nocedal and S.~J. Wright.
\newblock \emph{Numerical Optimization}.
\newblock Springer, 2006.

\bibitem[Pearlmutter(1994)]{Pearlmutter94}
B.~A. Pearlmutter.
\newblock Fast exact multiplication by the {H}essian.
\newblock \emph{Neural computation}, 6\penalty0 (1):\penalty0 147--160, 1994.

\bibitem[Regier et~al.(2017)Regier, Jordan, and McAuliffe]{ReigerJoMcAu17}
J.~Regier, M.~I. Jordan, and J.~McAuliffe.
\newblock Fast black-box variational inference through stochastic trust-region
  optimization.
\newblock In \emph{Advances in Neural Information Processing Systems 31}, 2017.

\bibitem[Schraudolph(2002)]{Schraudolph02}
N.~N. Schraudolph.
\newblock Fast curvature matrix-vector products for second-order gradient
  descent.
\newblock \emph{Neural computation}, 14\penalty0 (7):\penalty0 1723--1738,
  2002.

\bibitem[Simchowitz(2018)]{Simchowitz18}
M.~Simchowitz.
\newblock On the randomized complexity of minimizing a convex quadratic
  function.
\newblock \emph{arXiv:1807.09386 [cs.LG]}, 2018.

\bibitem[Trefethen and Bau~III(1997)]{TrefethenBa97}
L.~N. Trefethen and D.~Bau~III.
\newblock \emph{Numerical Linear Algebra}.
\newblock SIAM, 1997.

\bibitem[Tripuraneni et~al.(2017)Tripuraneni, Stern, Jin, Regier, and
  Jordan]{TripuraneniStJiReJo17}
N.~Tripuraneni, M.~Stern, C.~Jin, J.~Regier, and M.~I. Jordan.
\newblock Stochastic cubic regularization for fast nonconvex optimization.
\newblock \emph{arXiv:1711.02838 [cs.LG]}, 2017.

\bibitem[Tseng(2008)]{Tseng08}
P.~Tseng.
\newblock On accelerated proximal gradient methods for convex-concave
  optimization.
\newblock 2008.
\newblock URL \url{http://www.mit.edu/~dimitrib/PTseng/papers/apgm.pdf}.

\bibitem[Yao et~al.(2018)Yao, Xu, Roosta-Khorasani, and Mahoney]{YaoXuRoMa18}
Z.~Yao, P.~Xu, F.~Roosta-Khorasani, and M.~W. Mahoney.
\newblock Inexact non-convex newton-type methods.
\newblock \emph{arXiv:1802.06925 [math.OC]}, 2018.

\bibitem[Zhang et~al.(2017)Zhang, Shen, and Li]{ZhangShLi17}
L.-H. Zhang, C.~Shen, and R.-C. Li.
\newblock On the generalized {L}anczos trust-region method.
\newblock \emph{SIAM Journal on Optimization}, 27\penalty0 (3):\penalty0
  2110--2142, 2017.

\end{thebibliography}
